\documentclass[a4paper, 11pt]{article}

\usepackage{amsmath,amssymb,amsfonts,amsthm,subfigure}
\usepackage{xcolor}
\usepackage{graphicx}
\usepackage{mathtools}
\usepackage{amssymb}
\usepackage{url}
\usepackage{amsbsy}
\usepackage{epstopdf}
\usepackage{bbold}
\usepackage{microtype}
\usepackage[english]{babel}
\usepackage[T1]{fontenc}
\usepackage{ae}
\usepackage{caption}
\usepackage[a4paper, margin=2.15cm]{geometry}
\mathtoolsset{showonlyrefs}

\theoremstyle{plain}
\newtheorem{prop}{Proposition}[section]

\newtheorem{thm}[prop]{Theorem}
\newtheorem{lem}[prop]{Lemma}
\theoremstyle{remark}
\newtheorem{rem}{Remark}

\allowdisplaybreaks

\newcommand{\R}{{\mathbb{R}}}
\newcommand{\Dtt}{D_{\mathrm{tt}}}
\newcommand{\Dt}{D_{\mathrm{t}}}

\newcommand{\M}{\mathcal{M}}
\newcommand{\N}{\mathcal{N}}
\newcommand{\T}{\mathrm{T}}
\renewcommand{\S}{\mathcal{S}}
\newcommand{\Id}{\mathrm{Id}}
\newcommand{\IdM}{{\mathbb{1}}}
\newcommand{\SO}{\mathrm{SO}}
\renewcommand{\O}{\mathrm{O}}

\renewcommand{\P}{\mathrm{P}}
\newcommand{\Q}{\mathrm{Q}}
\newcommand{\dist}{{\bf{d}}}
\newcommand{\n}{{\bf{n}}}
\newcommand{\match}{\mathrm{match}}
\newcommand{\mem}{\mathrm{mem}}
\newcommand{\bend}{\mathrm{bend}}
\newcommand{\vol}{\mathrm{vol}}
\newcommand{\diag}{\mathrm{diag}}

\newcommand{\eps}{\varepsilon}
\newcommand{\wkto}{\rightharpoonup}
\DeclareMathOperator{\Div}{div}
\DeclareMathOperator{\cof}{Cof}

\DeclareMathOperator{\tr}{tr}
\DeclareMathOperator{\supp}{supp}
\DeclareMathOperator{\diam}{diam}

\DeclareMathOperator{\dd}{d\!}

\newcommand{\gs}{\geqslant}
\newcommand{\ls}{\leqslant}
\newcommand{\Srt}{{\mathcal{S}_1^{\frac{1}{2}}}}
\newcommand{\SrtI}{{\mathcal{S}_1^{-\frac{1}{2}}}}
\newcommand{\SSrt}{{\mathcal{S}_2^{\frac{1}{2}}}}
\newcommand{\SSrtI}{{\mathcal{S}_2^{-\frac{1}{2}}}}
\renewcommand{\n}{{\bf{n}}}
\newcommand\restr[2]{{
  \left.\kern-\nulldelimiterspace 
  #1 
  \vphantom{\big|} 
  \right|_{#2} 
  }}

\begin{document}
\title{Symmetry and scaling limits for matching of implicit surfaces based on thin shell energies}
\author{Jos\'{e} A.~Iglesias\thanks{Johann Radon Institute for Computational and Applied Mathematics (RICAM), Austrian Academy of Sciences, Linz, Austria (\texttt{jose.iglesias{@}ricam.oeaw.ac.at})}}
\date{}
\maketitle

\begin{abstract}
In a recent paper by Iglesias, Rumpf and Scherzer (Found. Comput. Math. 18(4), 2018) a variational model for deformations matching a pair of shapes given as level set functions was proposed. Its main feature is the presence of anisotropic energies active only in a narrow band around the hypersurfaces that resemble the behavior of elastic shells. In this work we consider some extensions and further analysis of that model. First, we present a symmetric energy functional such that given two particular shapes, it assigns the same energy to any given deformation as to its inverse when the roles of the shapes are interchanged, and introduce the adequate parameter scaling to recover a surface problem when the width of the narrow band vanishes. Then, we obtain existence of minimizing deformations for the symmetric energy in classes of bi-Sobolev homeomorphisms for small enough widths, and prove a $\Gamma$-convergence result for the corresponding non-symmetric energies as the width tends to zero. Finally, numerical results on realistic shape matching applications demonstrating the effect of the symmetric energy are presented.
\end{abstract}

\section{Introduction}
We are interested in variational methods for the matching of implicit shapes, in which an energy for deformations defined in a computational domain containing both shapes is minimized. More specifically, we are given two embedded $C^2$ diffeomorphic hypersurfaces $\M_1, \M_2 \subset \Omega \subset \R^d$, where $\Omega$ is an open bounded domain with Lipschitz boundary, and we work with models formulated through the signed distance functions $\dist_i$ to $\M_i$. The matching is then accomplished through a deformation $\phi: \Omega \to \Omega$ such that $\phi(\M_1) \approx \M_2$ and with the aim that perceptually similar regions of $\M_1$ and $\M_2$ correspond to each other. The particular notion of similarity we use is derived from variational integrals penalizing distortion along the tangent spaces of the $\M_i$, and mismatch of their curvatures in a tensorial fashion through their shape operators.

In this context, we say that an energy is \textit{symmetric} if it assigns the same value to a deformation for matching two shapes and to the inverse of the deformation when matching the shapes in the opposite order. This kind of consistence is not at all guaranteed when formulating such a model, yet it is often desirable. Besides basic conceptual reasons, many applications of statistical analysis like Fr\'echet means or PCA on spaces of shapes are based on similarity measures. One possible choice (see the overview \cite{RuWi10}) are those based on deformation energies, in which case symmetry is clearly advantageous. Another particular situation where such symmetry would be desirable is the time-discrete geodesic calculus for shapes \cite{RumWir13, RumWir15}, a framework in which a deformation energy can be used to induce a Riemannian distance. In that case one expects the continuous geodesics being approximated to be invariant with respect to time reversal, and a symmetric energy ensures this reversibility already on the discrete level.

We introduce in Section \ref{subsec:symenergy} a new symmetric energy consisting of a matching penalization term for the constraint $\phi(\M_1)=\M_2$, a membrane term that penalizes tangential distortion, a bending-like term that induces curvature matching of the initial and target hypersurfaces, and an additional regularization based on an hyperelastic bulk energy. This structure was also used in the less refined energies already proposed in \cite{IglBerRumSch13} and \cite{IglRumSch18}. Each term of the new energy is symmetric with respect to switching the hypersurfaces with each other and the deformation for its inverse. Moreover, the first three energy contributions arise only from narrow bands $\mathcal{N}_\sigma \M_i = \{ x \in \Omega \,|\, -\sigma < \dist_i(x) < \sigma \}$, as an approximation of their influence only on the hypersurfaces to be matched.

These membrane and bending-like energies are centered around the projected tangential derivative construction introduced in \cite{IglRumSch18}, which is specific to level set matching. By considering the deformed area only along the tangent spaces of the offsets to the target surface, it allows for relaxing the constraint $\phi(\M_1)=\M_2$ while avoiding oscillations that would arise when attempting to keep the deformations fully isometric \cite[Sec.~4.1]{IglRumSch18}. This derivative is composed with explicit bounded, coercive, frame-invariant and isotropic stored energy functions which attain their global minimum at a single energy well in $\SO(d)$, a fact proved in Lemma \ref{lem:Wgeneral}. The membrane energy measures distortion of the projected tangential derivative through this stored energy function directly, while the bending-like term additionally uses anisotropy and non-identity resting configurations to penalize mismatch of curvatures of the $\M_i$ through $\phi$. This notion of projected tangential derivative is not just weakly continuous \cite[Lem.~4.1]{IglRumSch18} but in fact gives rise to polyconvex energy densities, as we show in Lemma \ref{lem:polyconvexity}. Combined with an a priori estimate given in Lemma \ref{lem:distest} for the maximum mismatch of the shapes in terms of the strength of the matching penalization, these lower semicontinuity properties are used in Theorem \ref{thm:existence} to prove existence of minimizers in classes appropriate to the symmetry with respect to inversion, that is, bi-Sobolev deformations.

An obvious price that is paid to work in the level set framework is the increase of dimension of the domain, and this is equally true for the nonlinear, thin-shell based matching energies used in \cite{IglBerRumSch13, IglRumSch18} and for the current work. We aim to offer further theoretical justification for this family of matching energies by studying in Theorem \ref{thm:gammaconv} the membrane limits of a non-symmetric version of the energy as the thickness of the narrow bands tends to zero and the matching penalization becomes exact, so that the resulting energy has terms defined purely on the hypersurfaces. In this situation, the projected tangential derivative trivializes the quasiconvexification usually appearing in this kind of limit (the membrane energy of \cite{LedRao95}), so the structure of the energies used is preserved. A limitation is that we are only able to perform this asymptotic analysis for energies that do not enforce injectivity of the deformations, with the consequence that the new symmetric energy is not covered. The development of the tools that would be needed to naturally derive this kind of results with injective Sobolev deformations is a major problem in the theory of nonlinear elasticity, with partial solutions available only in two dimensions (see the end of Section \ref{sec:limits} for some discussion).

On the numerical level, the increase of dimension is mitigated by the use of multiscale descent schemes on adaptive meshes which are subdivided only around the input surfaces or curves. We present in Section \ref{sec:numerics} numerical examples computed with such a method for the new symmetric energies, showing a marked improvement in symmetry with respect to a non-symmetric version of the energy. These computations are based on a linear finite element discretization on octree grids where each cube is divided into tetrahedra. Such grids allow for fast indexing of degrees of freedom, indispensable for the use of coefficients depending on the deformed configuration, which is pervasive in our definition of the energy.

\subsection{Related work}
Our main focus is the formulation of symmetric energies, as defined above. The use of such energies for image registration for medical image registration was proposed already in \cite{CacRey00}. More recently, distances based on symmetrized hyperelastic volume energies (without tangential terms) were used for the analysis of cell shapes extracted from fluorescence microscope images in \cite{KouSleRoh15}. Outside imaging applications, the use of symmetric energies for modelling of nonlinear elasticity is advocated in \cite{IwaOni09}.

Our formulations have some common points with the modelling of thin shells through signed distance functions in \cite{DelZol95}. Thin structures and dimension reduction are a foundational topic in mathematical elasticity, treated by a vast number of works. On a general level, we mention only the book \cite{Cia00} for a thorough introduction to the modelling and analysis of shell problems, and \cite{FriJamMul06} as a starting point for the literature on nonlinear scaling limits obtained by $\Gamma$-convergence. The main techniques we use for our dimension reduction result arise from the membrane problem \cite{LedRao95, LedRao96} and problems of thin inclusions or `welding' \cite{AceButPer88, BesKraMic09}.

A number of works deal with shape analysis tasks using formulations based in linearized elasticity, like \cite{FuJuScYa09}. Shape matching using nonlinear thin shell energies has been tackled for parametric domains in \cite{DrLiRuSc05} and for triangulated surfaces in \cite{WiScSc11, EzuEtAl19}. Some precedents for shape analysis based on signed distance functions are \cite{DelZol94} and \cite{CharFauFer05}.

Another prominent body of work in mathematical shape analysis is that dealing with shape spaces from an intrinsic Riemannian perspective \cite{BauBruMic14}. This point of view has recently \cite{BauBruChaMol19, BauChaHarHsi20} been combined with varifold similarity metrics for shape matching without needing to estimate reparametrizations.

Our models are based on polyconvex energy functions, and there are also a number of works applying these for shape averaging \cite{RuWi08}, image registration \cite{DroRum04, BurModRut13}, or as part of joint registration/segmentation models \cite{DebEtAl20}.

\subsection{Notation}\label{subsec:notation}
\begin{itemize}
\item The euclidean inner product of two vectors $v,w \in \R^d$ is denoted by $v \cdot w$, and the Frobenius inner product of two square matrices $A, B \in \R^{d \times d}$ by $A:B=\tr(A^\T B)$. In both cases, $|v|$ or $|A|$ denotes the corresponding norm induced by these inner products. We denote the tensor product of $v,w \in \R^d$ by $v \otimes w = v w^\T \in \R^{d \times d}$.
\item $\Omega \subset \R^d$ is a bounded domain, with strongly Lipschitz boundary (that is, it can be locally expressed as the graph of a Lipschitz function). For scalar functions $u: \Omega \to \R$ we denote by $\nabla u$ their usual gradient and by $D^2 u$ the Hessian matrix, while for vector fields $\phi: \Omega \to \R$ we denote the Jacobian matrix by $(D\phi)_{ij} = \partial_{j} \phi^{i}$.
\item The identity function is denoted by $\Id$, whereas $\IdM \in \R^{d \times d}$ stands for the identity matrix.
\item For $i=1,2$, $\M_i \subset \Omega$ are compact $C^2$ hypersurfaces diffeomorphic to each other, and $\dist_i$ denote the signed distance to them, with the convention that these are negative in the interior components induced by $\M_i$. With $\n_i(x) := \nabla \dist_i(x)$ we denote the outer normal vectors to the offset hypersurfaces $\{y \,|\, \dist_i(y)=\dist_i(x)\}$ of $\M_i$, and by $\P_i:=\IdM - \n_i \otimes \n_i$ the orthogonal projections onto the corresponding tangent spaces. 
\item Noticing that the shape operators of the offset hypersurfaces to $\M_i$ can be read off $D^2 \dist_i$ (see \cite[Lem.~14.17]{GilTru01}), we use the notation $\S_i := \mathcal{R} (D^2 \dist_i + \n_i \otimes \n_i)$ for uniformly positive definite matrices derived from them, where $\mathcal{R}$ is a regularized absolute value function for matrices discussed in Section \ref{subsec:tangential}.
\item $\N_r \M_i := \{x \in \R^d \mid |\dist_i(x)| \leq r \}$ denote tubular neighborhoods of width $r$ of $\M_i$.
\item Occasionally we write $\Dt\phi := D\phi\, \P_1$ for the standard tangential derivative along the tangent spaces of the offsets to $\M_1$, while $\Dtt\phi := (\P_2\, \circ \phi) D\phi\, \P_1 + (\n_2 \circ \phi) \otimes \n_1$ is the projected tangential derivative (see Section \ref{subsec:tangential}) for measuring tangential distortion of a deformation $\phi$ attempting to match $\M_1$ onto $\M_2$.
\item $\Lambda[M, N, A, v, w]:=\P_2 N^{\frac{1}{2}} \P_2 A \P_1 M^{-\frac{1}{2}} \P_1 + w \otimes v$ for $A \in R^{d \times d}$ arbitrary, $M, N \in \R^{d \times d}$ symmetric positive definite, $v,w \in \R^d$ are classifier matrices for the purpose of curvature matching (when applied to $\S_i, D\phi$ and $\n_i$, see Section \ref{subsec:tangential}).
\item For a given unit vector $e \in \R^d$ we denote by $\Q(e)\in \SO(d)$ any proper rotation such that $\Q(e) e_d = e$, where $e_d$ denotes the $d$-th element of the canonical basis of $\R^d$. This condition does not specify a unique $\Q(e)$, but the properties above will be the only ones used for $Q$.
\item Deformations considered as candidates for matching $\M_1$ to $\M_2$ are usually denoted by $\phi:\Omega \to \Omega$, while `inverse' deformations that should match $\M_2$ to $\M_1$ are denoted by $\psi$. 
\item $C$ denotes an unspecified positive constant, which could be different in each appearance, even inside the same line.
\end{itemize}

\section{Symmetric level set matching energies}\label{sec:energies}

We aim to formulate a matching energy which is symmetric with respect to simultaneously swapping the input shapes and taking the inverse of the deformation. To this end, we consider explicit penalization of the inverse deformations in all of the energy terms. Our starting point is the observation that for regular enough deformations, integral energies associated to the inverse deformation can be computed in the original domain through a change of variables.

Let $p>d$ and $\phi \in W^{1,p}(\Omega; \R^d)$ be such that its continuous representative is an homeomorphism and $\phi^{-1} \in W^{1,p}(\Omega; \R^d)$ (i.e., $\phi$ is $p$-bi-Sobolev). Since $p>d$, $\phi$ has the Lusin N-property \cite[Theorem 4.2]{HenKos14}, that is, it maps sets of zero measure to sets of zero measure. Therefore, we can use the change of variables formula \cite[Theorem A.35]{HenKos14}, so that applying the chain rule and Cramer's rule we end up with:
\begin{equation}\label{eq:inverseIntegral}\int_{\phi(\Omega)} F\big(y, \phi^{-1}(y), D(\phi^{-1})(y)\big) \dd y = \int_\Omega F\left(\phi(x), x, \frac{\cof D\phi(x)^\T}{\det D\phi(x)}\right) |\det D\phi(x)| \dd x.\end{equation}
for any Carath\'{e}odory integrand $F:\Omega \times \R^d \times \R^{d \times d} \to \R$.

\subsection{Symmetric energy functional}\label{subsec:symenergy}
We now formulate the different terms of our energy. Let
\[\eta \in C^1_0(\R;\R), \text{ with } \int_\R \eta=1 \text{ and } \supp \eta = [-1, 1]\]
and define
\[\eta_\sigma(s): = \frac{1}{\sigma}\,\eta\left(\frac{s}{\sigma}\right)\text{, so that }\int_{\R} \eta_\sigma  = 1 \text{ for all }\sigma.\]
One option would be to choose $\eta \in C^\infty_0(\R;\R)$ as for standard mollifiers, but we only need one derivative for our first-order numerical descent. Moreover, choosing $\eta$ of polynomial decay allows for more detailed estimates, which are required for existence of minimizers with weights given as powers of $\sigma$ in the constraint penalty term in \eqref{eq:EsymMatch} and vanishing volume regularization \eqref{eq:EsymVol} below.

Choosing our main parameter for scaling to be the size $\sigma$ of the narrow band, we introduce two scaling exponents. The first is denoted by $q$ and controls how intensely the matching penalty is enforced. The second, denoted by $\theta \in \{0,1\}$, controls the behaviour of the volume term.  Our complete energy, taking into account contributions of the inverse map for each term through \eqref{eq:inverseIntegral} reads
\begin{equation}\label{eq:Esym}
E^\sigma[\phi]:=E^\sigma_{\match}[\phi]+E^\sigma_{\mem}[\phi]+E^\sigma_{\bend}[\phi]+E^\sigma_{\vol}[\phi], \text{where }
\end{equation}
\begin{equation}\label{eq:EsymMatch}
E^\sigma_{\match}[\phi]:=\frac{1}{\sigma^q}\int_\Omega \Big( \eta_\sigma(\dist_1) + \eta_\sigma(\dist_2 \circ \phi)\, \big|\det D\phi\big| \Big) |\dist_2 \circ \phi - \dist_1|^2,
\end{equation}
\begin{equation}\label{eq:EsymMem}
\begin{aligned}
E^\sigma_{\mem}[\phi]:=&\int_\Omega \eta_\sigma(\dist_1) W\big( (\P_2 \circ \phi) D\phi\, \P_1  + (\n_2 \circ \phi) \otimes \n_1 \big)\\
&\enskip+\eta_\sigma(\dist_2 \circ \phi) W\!\left( \P_1 \frac{\cof D\phi^\T}{\det D\phi} (\P_2 \circ \phi)  + \n_1 \otimes (\n_2 \circ \phi) \right)\big|\det D\phi\big|,
\end{aligned}
\end{equation}
\begin{equation}\label{eq:EsymBend}
\begin{aligned}
E^\sigma_{\bend}[\phi]:=&\int_\Omega \eta_\sigma(\dist_1) W\big( \Lambda[\S_1, \S_2 \circ \phi, D\phi, \n_1, \n_2 \circ \phi] \big)\\
&\enskip+\eta_\sigma(\dist_2 \circ \phi) W\!\left( \Lambda\left[\S_2 \circ \phi, \S_1, \cof D\phi^\T / \det D\phi, \n_2 \circ \phi, \n_1\right] \right)\big|\det D\phi\big|,
\end{aligned}
\end{equation}
\begin{equation}\label{eq:EsymVol}
E^\sigma_{\vol}[\phi]:=\sigma^\theta\int_\Omega W(D\phi)+W\left(\frac{\cof D\phi^\T}{\det D\phi}\right)\big|\det D\phi\big|.
\end{equation}
Here, $W:\R^{d \times d} \to \R$ is a $p$-coercive and polyconvex (that is, it can be written as a jointly convex function of the matrix argument and determinants of its minors of any order \cite[Def.~5.1(iii)]{Dac08}) stored energy function minimized at $\SO(d)$, whose specific form is discussed in Section \ref{subsec:storedenergy}. The form of the first terms in $E^\sigma_{\mem}$ and $E^\sigma_{\bend}$ follows the constructions introduced in \cite{IglRumSch18}. We have postponed the definition of $\Lambda$ and $\S_i$ to Section \ref{subsec:tangential} below, where we also recall the motivation for these formulas. 

In case $\theta=0$ the volume term is equally strong as $\sigma \to 0$, interfering with the surface terms. In Section \ref{sec:limits} we consider the $\Gamma$-limit as $\sigma \to 0$ of a non-symmetric version (without the inverse terms) of the functional in this regime. In contrast if $\theta=1$ the volume term does not interfere in the limit, but uniform $W^{1,p}$ bounds on the corresponding minimizers are lost, complicating the ensuing analysis.

For practical applications each term can be multiplied by a positive constant $c_{\match}$,$c_{\mem}$, $c_{\bend}$, $c_{\vol}$ to balance the relative strength of each effect; we will do so for our numerical examples in Section \ref{sec:numerics}, but skip these in the rest of the presentation to not further complicate the notation.

Notice that since in the volume energy we are using the energy on both the deformation and its inverse via $W(D\phi)+W(\cof D\phi^\T / \det D\phi)|\det D\phi|$, no injectivity penalization is  needed in $W$ itself, that is $W(A)$ can remain bounded as $\det A \to 0$. Nevertheless, it only makes sense to consider this energy when $\det D\phi > 0$ almost everywhere. This property is satisfied by deformations belonging to the class that we consider in Section \ref{sec:existence}, see \eqref{eq:biSobolevDir}.

\subsection{Projected tangential derivatives and curvature classifiers}\label{subsec:tangential}
One of the main novelties of \cite{IglRumSch18} is measuring tangential distortion through the first term of \eqref{eq:EsymMem}, using the projected tangential derivative $(\P_2 \circ \phi) D\phi\, \P_1  + (\n_2 \circ \phi) \otimes \n_1$. This can be seen as a relaxation of physical models of tangential distortion energies, which is specific to shape matching of hypersurfaces given as level sets. This is because it utilizes the projection $\P_2 = \IdM - \nabla \dist_2 \otimes \nabla \dist_2$ to the tangent space to the target hypersurface, evaluated at the point $\phi(x)$ which may not necessarily lie exactly on $\M_2$, so the signed distance function is needed to obtain a surrogate of the geometry from it. In any case, if we had $\phi(\M_1)=\M_2$, the second projection would be superfluous and this construction would measure tangential distortion exactly. Here we use the same construction, with the addition of the symmetrized term which accounts for tangential distortion, in the same projected sense, but for the inverse of the deformation that should match $\M_2$ onto $\M_1$. Further details and explanations, along with comparison with constructions based on the plain tangential derivative are given in \cite[Secs.~2.1, 3.1, 4.1]{IglRumSch18}.

We remark that it is possible for a point $x \in \Omega$ to simultaneously satisfy
\[\det D\phi(x) >0 \text{ and }\det \big( \P_2(\phi(x)) D\phi(x)\, \P_1(x)  + \n_2(\phi(x)) \otimes \n_1(x) \big) <0,\] 
depending on the relative positions of $\phi(\M_1)$ and $\M_2$. As a simple example, consider $\phi$ to be the identity map in $\Omega=(-3,3)^2$ with $\M_1=\mathbb{S}^1+(1,0)$ and $\M_2=\mathbb{S}^1-(1,0)$, for $\mathbb{S}^1$ the unit circle. In this case, the projected tangential derivative at the origin turns out to be $-e_1 \otimes e_1 + e_2 \otimes e_2$, where $e_i$ are the standard cartesian unit vectors. This matrix is orientation reversing, the reason being that the tangent spaces are mapped to each other in reverse orientation. Of course, when mapping though a homeomorphism which nearly matches $\M_1$ to $\M_2$ this situation would seldom happen, and when exactly mapping $\M_1$ to $\M_2$ it cannot happen at all, but this cannot be enforced for all the iterates computed in a numerical descent. Therefore, it is paramount that the energy density used in $W$ is defined and finite on all of $\R^{d \times d}$ regardless of orientation, while being minimized at least locally at $\SO(d)$. The specific density \eqref{eq:Wbounded} we use for numerical computations satisfies these conditions along with additional continuity properties.

Turning our attention to the bending-like energy $E_\bend^\sigma$ in \eqref{eq:EsymBend}, we first define
\[\S_i(x) := \mathcal{R}\big(D^2 \dist_i + \n_i(x) \otimes \n_i(x)\big),\]
where $\mathcal{R}:\R^{d \times d} \to \R^{d \times d}$ is a regularization operator defined below, and $D^2 \dist_i + \n_i(x) \otimes \n_i(x)$ is a nonsingular matrix that reflects the shape operator to the offset hypersurface of $\M_i$ at the point $x$ (that is, $\{y \,|\, \dist_i(y)=\dist_i(x)\}$) when restricted to its tangent space, and with the normal direction $\n_i(x)$ as an eigenvector with unit eigenvalue. These are used in the classifier matrix introduced in \cite{IglRumSch18} and given for symmetric matrix fields $M,N$ and arbitrary square matrix fields $A$ by
\begin{equation}\label{eq:lambdadef}\Lambda[M,N,A,\n_1,\n_2] := \P_2 N^{\frac12} \P_2 A \P_1 M^{-\frac12} \P_1 + \n_2 \otimes \n_1.\end{equation}
It can be seen through a relatively straightforward computation (see \cite[Lem.~3.1]{IglRumSch18}) that whenever $A\in \R^{d \times d}$ satisfies $A \P_1 = \P_2 A$ and $M,N \in \R^{d \times d}$ are symmetric positive definite matrices for which
\[M=\P_1 M \P_1 + \n_1 \otimes \n_1 \text{ and } N=\P_2 N \P_2 + \n_2 \otimes \n_2,\]
then the following two conditions are equivalent:
\begin{equation}\label{eq:curvaturematching}\begin{gathered}
A^T \P_2 N \P_2 A =  \P_1 M \P_1,\text{ and }\\
\Lambda[M,N,A,\n_1,\n_2] = \P_2 N^{\frac12} \P_2 A \P_1 M^{-\frac12} \P_1 + \n_2 \otimes \n_1 \in \O(n).
\end{gathered}\end{equation}
In the above (for the case $A=D\phi$) we recognize the first equation as the transformation rule for second-order tensors defined at the tangent spaces $\nabla \dist_i^\perp$, such as the shape operators of the hypersurfaces $\M_i$. The second conditions implies $W(\Lambda[M,N,A, \n_1, \n_2])$ is pointwise minimized, since we assume it has an energy well at $\SO(d)$. Therefore, the integrands of $E_\bend^\sigma$ in \eqref{eq:EsymBend} can be seen as multiplicatively measuring the failure of $\S_2$ to be pulled back to $\S_1$. This can also  be seen as a relaxed matching condition that would resemble a true bending energy whenever $\phi(\M_1)=\M_2$, but that doesn't take into account the curvature of $\phi(\M_1)$ directly and uses the one of $\M_2$ instead. 

A limitation is that the equivalence of \eqref{eq:curvaturematching} is only valid whenever $M,N$ are positive definite. For this purpose use a regularized absolute value function for the eigenvalues of symmetric matrices. Fixing $i=1$ for concreteness and assuming the matrix $D^2 \dist_1(x) + \n_1(x) \otimes \n_1(x)$ can be diagonalized as $Q(x)^T \mathrm{diag}(\lambda_1(x),\ldots, \lambda_d(x)) Q(x)$ where $Q(x) \in \SO(d)$ for each $x\in \Omega$, we define 
\begin{equation}\label{eq:defSi}\S_1(x) := \mathcal{R}\big(D^2 \dist_1(x) + \n_1(x) \otimes \n_1(x)\big) := Q(x)^T \mathrm{diag}\big(\max(|\lambda_1(x)|,\tau),\ldots, \max(|\lambda_n(x)|,\tau)\big) Q(x)\end{equation}
where $\tau >0$ is a small positive parameter. This means that although sensitive to curvature directions and magnitudes, our matching conditions must be agnostic to the signs of the curvatures. Although this limits the capacity of $\Lambda[\S_1, \S_2 \circ \phi, D\phi, \n_1, \n_2 \circ \phi]$ to enforce correct curvature matching since it might identify saddle points with elliptical ones, this term still helps to align the hypersurfaces through its tensorial character. For further information about this method of first-order curvature matching we refer again to \cite[Secs.~2.2 and 3.2]{IglRumSch18}.

\subsection{Stored energy functions}\label{subsec:storedenergy}
The integrands $F$ for our energy are constructed from a polyconvex stored energy function $W:\R^{d \times d} \to \R$, such that $W \gs 0$, $W(A)=0$ if $A \in \SO(d)$ such that $W(A) \gs C |A|^p$ for some $p > d$. When introducing specific examples below we take $p=d+1$, for simplicity in the formulas. Let us also reiterate that $W$ is required to be defined on all of $\R^{d\times d}$, and not only for $A$ with $\det A > 0$. A particularly compact such function with appropriate coercivity, inspired by the ones used in \cite{IglRumSch18}, is given in any dimension $d\gs 2$ by
\[W_o(A):=\frac{1}{d+1}|A|^{d+1} + d^{\frac{d-1}{2}}e^{1-\det A} - \frac{1}{d+1} d^{\frac{d+1}{2}} - d^{\frac{d-1}{2}}.\]
In particular, for $d=3$,
\[W_o(A)=\frac{1}{4}|A|^{4}+3 e^{1-\det A}-21.\]
It can be checked that the above function attains a local minimum at $\SO(d)$ by rewriting it in terms of singular values, which is possible \cite[Prop.~5.31]{Dac08} because they are frame-invariant and isotropic.

A disadvantage of the above stored energy is that even though it is coercive in $W^{1,d+1}(\Omega; \R^d)$, due to the exponential term it does not satisfy bounds of the type $W_o(A)\ls C (1+|A|^{d+1})$, which will be required in the analysis of Section \ref{sec:limits}. Through the following lemma we can easily produce more suitable stored energy functions:

\begin{lem}\label{lem:Wgeneral}Let $d \gs 2$ and $\widehat W: (\R^+\cup \{0\}) \times \R \to \R$ be convex, increasing in its first argument, with $\widehat W(s, -t) > \widehat W(s, t)$ for any $s,t>0$ and such that the function $t \mapsto \widehat W(d^{d/2} t, t)$ attains its minimum at $t=1$. Then, the stored energy function $W: \R^{d\times d}\to \R^+\cup \{0\}$ defined by
\begin{equation}\label{eq:Wgeneral}W(A):=\widehat W\left(|A|^d, \det A\right)\end{equation}
attains its global minimum at $\SO(d)$. Moreover, $W$ is polyconvex and frame-indifferent.
\end{lem}
\begin{proof}
Let $A \in \R^{d \times d}$ be arbitrary. Since $\widehat W(s, -t) > \widehat W(s, t)$ while 
\[\det\big(\diag(-1,1,\ldots, 1)\,A\big)=-\det A \text{, and }\big|\diag(-1,\ldots, 1)A\big|_d=|A|_d,\]
we may assume $\det A > 0$ when looking for a minimum point, so that $\det A = \prod_i s_i$, where $(s_1, \ldots, s_n)$ are the singular values of $A$. Using the arithmetic mean-geometric mean inequality on these singular values we obtain
\begin{equation}\label{eq:normdet}|A|^d=\tr\big(A^\T A\big)^{\frac{d}{2}}=\left(\sum_{i=1}^{d} s_i^2\right)^{\frac{d}{2}} \gs \left( d \left(\prod_{i=1}^d s_i^2\right)^{\frac{1}{d}} \right)^{\frac{d}{2}}=d^{\frac{d}{2}} \prod_{i=1}^d s_i = d^{\frac{d}{2}} \det A.\end{equation}
Combining \eqref{eq:normdet}, the monotonicity on the first argument, and the minimality property, we get
\begin{equation}\label{eq:globalminimum}W(A)=\widehat W\big(|A|^d, \det A\big)\gs \widehat W\big(d^{\frac{d}{2}} \det A, \det A\big) \gs \widehat W\big(d^{\frac{d}{2}},1\big) = W(\IdM),\end{equation}
where $\IdM \in \R^{d \times d}$ is the identity matrix. Polyconvexity follows since $\widehat W$ is convex and increasing in its first argument, so the composition with $|\cdot|^d$ is still convex. Frame invariance is immediate since the singular values of $A$ and $QA$ with $Q \in \SO(d)$ are equal.
\end{proof}

\begin{rem}Since $d \gs 2$ we have that in the definition \eqref{eq:Wgeneral}, $W$ is differentiable whenever $\widehat W$ is, which is clearly advantageous when choosing a numerical implementation. 
\end{rem}

A particular example which satisfies the hypothesis of Lemma \ref{lem:Wgeneral}, coercive in $W^{1, p}$ with $p = d+1$, nonnegative, vanishing at $\IdM$, satisfying a bound of the type $W(A)\ls C(1+|A|^p)$ and with continuous derivatives is
\begin{equation}\label{eq:Wbounded}\begin{gathered}
W(A)=\frac{1}{d+1}|A|^{d+1} + \sqrt{2}\,d^{\frac{d-1}{2}}\sqrt{1 + (\det A - 2)^2} - \frac{1}{d+1} d^{\frac{d+1}{2}} - 2\,d^{\frac{d-1}{2}},\text{ with }\\
\widehat W(s,t)=\frac{1}{d+1}\,s^{\frac{d+1}{d}} + \sqrt{2}\,d^{\frac{d-1}{2}}\,\sqrt{1 + (t - 2)^2} - \frac{1}{d+1} d^{\frac{d+1}{2}} - 2\,d^{\frac{d-1}{2}}.
\end{gathered}\end{equation}
In the analysis that follows we will use all of these properties, but not the specific form of $W$. For the numerical computations presented in Section \ref{sec:numerics}, the specific formula \eqref{eq:Wbounded} is used.

In light of \eqref{eq:inverseIntegral} one might wonder about the behaviour of the energy associated to the inverse deformation, expressed through \eqref{eq:inverseIntegral}. In fact, we have that if $W$ is polyconvex, $W \gs 0$ and $W(A)=0$ whenever $A \in \SO(d)$, then the function defined for $A$ with $\det A >0$ by
\[\mathcal {W}(A) := W(A^{-1})\big|\det A\big| = W(\cof A^\T / \det A)\big|\det A\big|\]
is also polyconvex, $\mathcal{W}\gs 0$ and $\mathcal{W}(A)=0$ if $A \in \SO(d)$. Polyconvexity is proved in \cite[Thm.~2.6]{Bal77b} and \cite[Prop.~1.1, Sec.~2.5]{IwaOni09}. The minimality property follows from the assumption $\det A > 0$ and the fact that $\SO(d)$ is a group, so $A \in \SO(d)$ if and only if $A^{-1} \in \SO(d)$.

\subsection{Properties of the energy}
In \cite[Lem.~4.1]{IglRumSch18} it is proved that the determinant of the projected tangential derivative $\det\big((\P_2 \circ \phi) D\phi\, \P_1  + (\n_2 \circ \phi) \otimes \n_1\big)$ is weakly continuous with respect to weak convergence in $W^{1,p}(\Omega; \R^d)$. The following algebraic lemma provides an easier route to lower semicontinuity:
\begin{lem}\label{lem:polyconvexity}The infinitesimal projected area distortion induced by the derivative of the inverse deformation can be computed as the quotient of the stretching along normals and the determinant of the Jacobian. In symbols, for $A \in R^{d \times d}$ arbitrary and $\P_i = \IdM - \n_i \otimes \n_i$ we have
\begin{equation}\label{eq:inversearea}\det\left(\P_1 A^{-1} \P_2 + \n_1 \otimes \n_2 \right)=\det\left(\P_1 \frac{\cof A^\T}{\det A} \,\P_2  + \n_1 \otimes \n_2 \right)=\frac{\n_2^\T A\, \n_1}{\det A}.
\end{equation}
Similarly, for the determinant of the projected tangential derivative we have
\begin{equation}\label{eq:tancof}\det\left(\P_2 A\, \P_1  + \n_2 \otimes \n_1 \right)=\n_2^\T\, \cof A\, \n_1.
\end{equation}
In consequence, both the integrands $F_{\mem}, \mathcal{F}_{\mem}:\Omega \times \R^d \times \{A \in \R^{d \times d} \,|\, \det A > 0\} \to \R$ defined by
\begin{gather}
\label{eq:Fmemdirect}F_{\mem}(x,v,A) := W\Big(\P_2(v) A \, \P_1(x)  + \n_2(v) \otimes \n_1(x)\Big) \text{ and }\\
\label{eq:Fmeminverse}\mathcal{F}_{\mem}(x,v,A):=W\Big(\P_1(x) \frac{\cof A^\T}{\det A} \P_2(v) + \n_1(x) \otimes \n_2(v)\Big)\big|\det A\big|
\end{gather}
are polyconvex in their last argument.

Furthermore, noticing that the $\S_i$ are positive definite by the regularization $\mathcal{R}$ applied to the shape operators in \eqref{eq:defSi}, one can define the regularized Gaussian curvatures $K_i \in \R^+$ by 
\begin{equation}\label{eq:regularizedGaussianCurvs}
K_i:=\n_i^\T \cof \S_i \,\n_i=\det \left(\frac{\n_i^\T \S_i \n_i}{\det \S_i}\right)^{-1},
\end{equation}
for which we have
\begin{equation}\label{eq:shapecof}
\begin{gathered}\det\left(\P_2 \SSrt \P_2 A\, \P_1 \SrtI \P_1  + \n_2 \otimes \n_1 \right)=K_1^{-\frac{1}{2}}K_2^{\frac{1}{2}}\, \n_2^\T \cof A\, \n_1
\end{gathered}
\end{equation}
and analogously
\begin{equation}\label{eq:invshapecof}
\begin{gathered}
\det\left(\P_1 \Srt \P_1 \frac{\cof A^\T}{\det A}\, \P_2 \SSrtI \P_2  + \n_1 \otimes \n_2 \right)=K_1^{\frac{1}{2}}K_2^{-\frac{1}{2}}\,\frac{\n_2^\T A \n_1}{\det A}.
\end{gathered}
\end{equation}
Thereby the energy densities for $E^\sigma_{\bend}$, defined by (c.f. \eqref{eq:EsymBend} and \eqref{eq:lambdadef})
\begin{gather}
\label{eq:Fbenddirect}F_{\bend}(x,v,A) := W\Big(\P_2(v) \SSrt(v) \P_2(v) A\, \P_1(x) \SrtI(x) \P_1(x)  + \n_2(v) \otimes \n_1(x) \Big) \text{ and }\\
\label{eq:Fbendinverse}\mkern-18mu\mathcal{F}_{\bend}(x,v,A):=W\Big(\P_1(x) \Srt(x) \P_1(x) \frac{\cof A^\T}{\det A}\, \P_2(v) \SSrtI(v) \P_2(v) + \n_1(x) \otimes \n_2(v)\Big)\big|\det A\big|,
\end{gather}
are also polyconvex in $A$ whenever $\det A >0$.
\end{lem}
\begin{proof}
To prove \eqref{eq:inversearea}, we first use Cramer's rule for $A$, yielding
\begin{equation*}
A=(A^{-1})^{-1}=\frac{\cof \left( A^{-1} \right)^\T}{\det A^{-1}}=\cof \left(A^{-1}\right)^\T\det A.
\end{equation*}
Taking transposes, multiplying by $Q(\n_2)$ (as defined in Section \ref{subsec:notation}) on the right and by $Q(\n_1)^\T$ on the left, and dividing by $\det A$,
\begin{equation}\label{eq:cofcof1}
\begin{aligned}
\frac{Q(\n_1)^\T A^\T Q(\n_2)}{\det A}&=Q(\n_1)^\T\cof \left(A^{-1}\right)Q(\n_2)\\
&=\cof \big(Q(\n_1)^\T\big) \cof \left( A^{-1}\right)\cof\big(Q(\n_2)\big)\\
&=\cof \left(Q(\n_1)^\T A^{-1} Q(\n_2)\right),
\end{aligned}
\end{equation}
where we have used that $Q(\n_i) \in \SO(d)$. Now, as also noticed in \cite[Eq.~2.3]{IglRumSch18}, for any square matrix $B$ we have
\[\det(\P_1 B \P_2 + \n_1 \otimes \n_2)=\det\Big(Q(\n_1)^\T\big( \P_1 B \P_2 + \n_1 \otimes \n_2\big)Q(\n_2)\Big)=\Big[\cof \Big(Q(\n_1)^\T B Q(\n_2)\Big)\Big]_{dd}\]
where $[\cdot]_{dd}$ denotes the last diagonal element. With $B=A^{-1}=\cof A^\T / \det A$, taking into account \eqref{eq:cofcof1} and since $Q(v)e_d =v$ we get
\begin{equation*}
\begin{aligned}
\det\left(\P_1 \frac{\cof A^\T}{\det A} \P_2 + \n_1 \otimes \n_2\right)&=\frac{\left[Q(\n_1)^\T A^\T Q(\n_2)\right]_{dd}}{\det A}=\frac{e_d^\T Q(\n_1)^\T A^\T Q(\n_2)e_d}{\det A}\\
&=\frac{\n_1^\T A^\T \n_2}{\det A}=\frac{\n_2^\T A\, \n_1}{\det A},
\end{aligned}
\end{equation*}
which is \eqref{eq:inversearea}. 

Next, interchanging the roles of $A$ and $A^{-1}$ and of $\n_1$ and $\n_2$, and again using Cramer's rule we obtain 
\[\det(\P_2 A \P_1 + \n_2 \otimes \n_1)=\frac{\n_2^\T A^{-T}\n_1}{\det A^{-1}}=\frac{\n_2^\T \cof A\,\n_1}{\det A^{-1}\det A}=\n_2^\T\, \cof A\,\n_1\]
which proves \eqref{eq:tancof}.

From \eqref{eq:tancof}, polyconvexity of $F_{\mem}$ is clear. Since $\mathcal{F}_{\mem}$ is the transformation of $F_{\mem}$ corresponding to the inverse deformation, the results of \cite{Bal77b, IwaOni09} again imply its polyconvexity.


Finally, for proving \eqref{eq:shapecof} one can write
\[\P_2 \SSrt \P_2 A\, \P_1 \SrtI \P_1  + \n_2 \otimes \n_1 = \big(\P_2 \SSrt \P_2 + \n_2 \otimes \n_2\big)\big(\P_2 A\, \P_1 + \n_2 \otimes \n_1 \big)\big(\P_1 \SrtI \P_1 + \n_1 \otimes \n_1\big),\]
take determinants on both sides, and use \eqref{eq:tancof} for each factor. Similarly, \eqref{eq:invshapecof} follows from \eqref{eq:inversearea}. The corresponding polyconvexity statements are then clear.
\end{proof}

\section{Existence of minimizers for symmetric matching energies}\label{sec:existence}
Consider the set of orientation-preserving bi-Sobolev homeomorphisms mapping $\Omega$ to itself:
\begin{equation}\label{eq:biSobolev}\mathcal{B}:=\left\{\phi \in W^{1,p}(\Omega; \R^d) \mid \phi(\Omega) = \Omega \text{ homeomorphically},\; \phi^{-1}\in W^{1,p}(\Omega; \R^d),\;\det D\phi > 0 \text{ a.e.}\right\},\end{equation}
and its subset with fixed identity Dirichlet (pure displacement) boundary conditions
\begin{equation}\label{eq:biSobolevDir}\mathcal{B}_0:= \mathcal{B} \cap \left( W_{0}^{1,p}(\Omega; \R^d) + \Id\right).\end{equation}
The discussion in the previous section suggests the latter as a natural space for posing our minimization problem. 

As in \cite{IglRumSch18}, we prove distance estimates ensuring that the image of a neighborhood of $\M_1$ can be forced to be uniformly close to $\M_2$ through the matching term, and vice versa. These guarantee that the deformed narrow band around $\M_1$ where the tangential terms are active remains in the part of the domain where $\dist_2$ is $C^2$, so that all the terms of the energy are well defined. However, compared to the situation in \cite{IglRumSch18} we need to keep a closer eye on the dependence on the parameters in the estimates. Whereas in that case $\sigma$ was fixed and one could choose a multiplicative parameter for $E_{\match}$ freely, here we couple these parameters with the prospect of considering the limit $\sigma \to 0$. A further difference is the case $\theta = 1$ which makes the volume term providing coercivity in $W^{1,p}(\Omega; \R^d)$ vanish as $\sigma \to 0$, which in turn affects how strongly the matching penalization must be enforced, as can be seen in condition \eqref{eq:coupledExponents}.

\begin{lem}\label{lem:distest}
Define \begin{equation}\label{eq:injectivradius}r_I:= \min\left( \frac{1}{\sup_{x \in \M_1}|D^2\dist_1(x)|}, \frac{1}{\sup_{x \in \M_2}|D^2\dist_2(x)|} \right),\end{equation}
and notice that $r_I >0$ since the $\M_i$ are $C^2$. Then there is $C_0 = C_0(\M_1, \M_2, \Omega)>0$ such that for all $\sigma \in (0, r_I)$ we have that
\begin{equation}\label{eq:energyboundwosigma}\inf_{\phi \in \mathcal{B}_0} E^\sigma[\phi] \ls C_0.\end{equation}
Moreover, assume that either
\begin{equation}\label{eq:decoupledExponents}\theta=0 \text{ and }q>0, \text{ or }\end{equation}
\begin{equation}\label{eq:coupledExponents}\theta =1,\ \eta \text{ is a spline of order }n\text{, and }q > n\,\max\left(\frac{1}{p-d} - 1, 0\right)+\frac{d}{p-d}-1.\end{equation}
Then for each $\eps >0$ there is some $\sigma_\eps = \sigma_\eps(\M_1, \M_2, \Omega, \theta, q)>0$ such that for all $0 < \sigma < \sigma_\eps$ and all $\phi$ with $E^\sigma[\phi] \ls C_0$ we have
\begin{equation}\label{eq:bandinclusion}\phi\left( \N_\sigma \M_1 \right) \subset \N_\eps \M_2 \text{, and } \phi^{-1}\left( \N_\sigma \M_2 \right) \subset \N_\eps \M_1,\end{equation}
where for $\delta > 0$ and $i=1,2$ we denote by $\N_\delta \M_i$ the tubular neighborhood $\{x \in \Omega \,|\, -\delta < \dist_i(x) < \delta\}$.
\end{lem}
\begin{proof}Since the hypersurfaces $\M_1$ and $\M_2$ are assumed to be diffeomorphic, let $\varphi:\M_1 \to \M_2$ be such a diffeomorphism. Now, the Frobenius norm $|D^2\dist_i(x)|$ is an upper bound for the principal curvatures of $\M_i$ at $x$, so that (see for example \cite[Lem.~6.3]{Mil63}) we may write each point $x \in \N_{r_I} \M_i$ as $x = y + t \n_i(y)$ with $y \in \M_i$ being the Euclidean projection of $x$ onto $\M_i$ and $|t| < r_I$. Using this notation we can extend $\varphi$ to a map $\varphi_{r_I}: \N_{r_I} \M_1 \to \N_{r_I} \M_2$ defined by $\varphi_{r_I}(y,t)=\varphi(y)+t \n_2(y)$ which is still a diffeomorphism. We then use the values of $\varphi_{r_I}$ at $\M_1 \pm r_I \n_1$ as Dirichlet boundary conditions for minimizers of a rescaled volume energy on the inside $\Omega_{\mathrm{in}}$ and outside $\Omega_{\mathrm{out}} = \Omega \setminus (\Omega_{\mathrm{in}} \cup \N_{r_I} \M_1)$ parts of the domain with respect to $\N_{r_I} \M_1$, that is
\begin{equation}\label{eq:insideprob}\inf_{\substack{\phi \in W^{1,p}(\Omega_{\mathrm{out}})\\\phi=\varphi+r_I \n_1 \text{ on }\M_1+r_I \n_1 \\ \phi = \Id \text{ on }\partial \Omega}} \int_{\Omega_{\mathrm{out}}} W(D\phi)+W\left(\frac{\cof D\phi^\T}{\det D\phi}\right)\big|\det D\phi\big|,\end{equation} 
and similarly for $\Omega_{\mathrm{in}}$ with boundary condition $\varphi-r_I \n_1 \text{ on }\M_1-r_I \n_1$ on $\partial \Omega_{\mathrm{in}}$. Piecing these three maps together, we obtain $\phi_I:\Omega \to \Omega$ for which $E_{\match}^\sigma[\phi_I] = 0$ for all $\sigma \in (0, r_I)$. Since the other terms \eqref{eq:EsymMem},\eqref{eq:EsymBend},\eqref{eq:EsymVol} of $E^\sigma[\phi_I]$ decrease as $\sigma \searrow 0$, as soon as $\phi_I \in \mathcal{B}_0$ we obtain the bound \eqref{eq:energyboundwosigma} with
\[C_0 := E_{\mem}^{r_I}[\phi_I]+E_{\bend}^{r_I}[\phi_I]+E_{\vol}^{r_I}[\phi_I].\] 
That $\det D\phi_I >0$ almost everywhere follows directly by its definition, since $|\partial \N_{r_I} \M_1|=0$, $\varphi_{r_I}$ is a diffeomorphism, and the energy density in \eqref{eq:insideprob} is unbounded as $\det D\phi \to 0$. By its definition in \eqref{eq:insideprob} $\phi_I$ belongs to $W_{0}^{1,p}(\Omega; \R^d) + \Id$. Moreover, since $\varphi_{r_I}$ is a $C^2$ diffeomorphism and the definition of $\phi_I$ in $\Omega_{\mathrm{in}}$ and $\Omega_{\mathrm{out}}$ we also have
\[\int_{\Omega} W(D\phi)+W\left(\frac{\cof D\phi^\T}{\det D\phi}\right)\big|\det D\phi\big| < +\infty,\]
which combined with $W(A)\gs C|A|^p$ and $p>d$, allows us to apply Ball's global invertibility theorem \cite[Thm.~2]{Bal81} to obtain that $\phi_I$ is a homeomorphism and $\phi_I^{-1} \in W^{1,p}(\Omega;\R^d)$.

Now we turn our attention to estimates for $\|\dist_2 \circ \phi \|_{L^\infty(N_\sigma\M_1)}$ (and for $\|\dist_1 \circ \phi^{-1} \|_{L^\infty(N_\sigma\M_2)}$, by symmetry) that allow us to conclude \eqref{eq:bandinclusion}. This is the same type of estimate proved in \cite[Eqs.~(4.15)-(4.23)]{IglRumSch18}, and its proof follows essentially the same steps, but since at present the strength of the matching and volume terms and the width of the narrow band are not independent of each other, we will have to be more precise. The strategy is to use the matching penalization term, which contains $|\dist_2 \circ \phi - \dist_1|^2$. However this function appears multiplied by the narrow band function $\eta_\sigma \circ \dist_1$, which decays to zero as $\dist_1 \nearrow \sigma$. To treat this difficulty, we introduce a cutoff width $\hat \sigma \in (0, \sigma)$ to split the narrow band in two parts to be estimated separately. First we notice that $\phi \in C^{0,\alpha}(\Omega)$ with $\alpha := 1-d/p$, by the Morrey inequality \cite[Thm.~7.17]{GilTru01}. Since the signed distance functions $\dist_i$ are $1$-Lipschitz, we have that
\begin{equation}\label{eq:distest1}\begin{aligned}
\|\dist_2 \circ \phi \|_{L^\infty(\{|\dist_1| \ls \sigma\})} &\ls 
\sigma + \|\dist_2 \circ \phi  - \dist_1 \|_{L^\infty(\{|\dist_1| \ls \sigma\})} \\
&\ls\sigma + \| \dist_2 \circ \phi  - \dist_1 \|_{L^\infty( \{ |\dist_1| < \sigma -\hat{\sigma}\} )} + |\dist_2 \circ \phi  - \dist_1|_{C^{0,\alpha}(\{\sigma - \hat{\sigma} \ls |\dist_1| \ls \sigma\})}\,\hat{\sigma}^\alpha \\
&\ls\sigma + \| \dist_2 \circ \phi  - \dist_1 \|_{L^\infty( \{ |\dist_1| < \sigma -\hat{\sigma}\} )} + \left(1 \!+\! | \phi |_{C^{0,\alpha}(\{\sigma - \hat{\sigma} \ls |\dist_1| \ls \sigma\})}\right)\hat{\sigma}^\alpha,
\end{aligned}\end{equation}
where $|\cdot|_{C^{0,\alpha}}$ denotes the H\"older seminorm (that is, $\|\cdot\|_{C^{0,\alpha}(A)} = |\cdot|_{C^{0,\alpha}(A)} +\sup_A |\cdot|$) and we have used that every point taken into account in the last term is at a distance less than $\hat{\sigma}$ from a point appearing in the second term. Moreover, we have assumed that $\sigma < 1$ to bring up the Lipschitz constant of $\dist_1$. Now, for the last term of \eqref{eq:distest1} we have, again by the Morrey inequality and using \eqref{eq:energyboundwosigma}, that
\begin{equation}\label{eq:distest2}\begin{aligned}\left(1 \!+\! | \phi |_{C^{0,\alpha}(\{\sigma - \hat{\sigma} \ls |\dist_1| \ls \sigma\})}\right)\hat{\sigma}^\alpha &\ls \left(1 \!+\! C\|D\phi \|_{L^p(\Omega)}\right)\hat{\sigma}^\alpha \\ &\ls C\left(1 \!+\! \big(\sigma^{-\theta} E_{\vol}^\sigma[\phi]\big)^{\frac{1}{p}} \right)\hat{\sigma}^\alpha \ls C\left(1 \!+\! \sigma^{-\frac{\theta}{p}} C_0^{\frac{1}{p}} \right)\hat{\sigma}^\alpha,\end{aligned}\end{equation}
for which if $\theta =1$ the right hand side can be made arbitrarily small by choosing $\hat{\sigma} = \sigma^r$ with $r > (\alpha p)^{-1} = 1/(p-d)$. Moreover, since we need to have $\sigma^r < \sigma$, also $r>1$ is required. In the case $\theta = 0$ any choice of $\hat{\sigma} < \sigma$ suffices. 

For the second term of \eqref{eq:distest1} we apply the Gagliardo-Nirenberg interpolation inequality (\cite[Thm.~5.8]{Ada03}, \cite[Thm.~1]{Nir66}) for a bounded domain $\Sigma$ and $u \in W^{1,p}(\Sigma)$
\begin{equation}\label{eq:interpolation}\|u\|_{L^\infty(\Sigma)} \ls C \left(\|\nabla u\|_{L^p(\Sigma)}^{\frac{d}{p}} \|u\|_{L^p(\Sigma)}^{1-\frac{d}{p}} + \|u\|_{L^p(\Sigma)}\right),\end{equation}
to $u=\dist_2 \circ \phi - \dist_1$ on the open set $\Sigma = \{ |\dist_1| < \sigma -\hat{\sigma}\}$. For the last term, using the monotonicity of $\eta_\sigma$, that $\sup_{\Omega}|\dist_i|\ls \diam \Omega$ and $\phi: \Omega \to \Omega$ we can estimate as $\sigma \to 0$
\begin{equation}\label{eq:distest3}\begin{aligned}
&\| \dist_2 \circ \phi  - \dist_1 \|_{L^p(\Sigma)} \ls \|\dist_2 \circ \phi - \dist_1 \|^{\frac{p-2}{p}}_{L^\infty( \Sigma )}\|\dist_2 \circ \phi  - \dist_1\|^{\frac{2}{p}}_{L^2( \Sigma )} \\
&\quad\qquad\ls \|\dist_2 \circ \phi - \dist_1 \|^{\frac{p-2}{p}}_{L^\infty( \Sigma )}\left(\big[\eta_\sigma(\sigma - \hat{\sigma})\big]^{-1}\int_\Sigma \left(\eta_\sigma \circ \dist_1\right) |\dist_2 \circ \phi -\dist_1|^2\dd x\right)^{\frac{1}{p}} \\
&\quad\qquad\ls  \big( 2 \diam \Omega \big)^{\frac{p-2}{p}}\Big(\big[\eta_\sigma(\sigma - \hat{\sigma})\big]^{-1} \sigma^q E_{\match}^\sigma[\phi]\Big)^{\frac{1}{p}} \\
&\quad\qquad\ls C\sigma^{\frac{q}{p}}E^\sigma[\phi]^{\frac{1}{p}} \big[\eta_\sigma(\sigma - \hat{\sigma})\big]^{-\frac{1}{p}}\\
&\quad\qquad\ls C\sigma^{\frac{q}{p}}C_0^{\frac{1}{p}}\big[\eta_\sigma(\sigma - \hat{\sigma})\big]^{-\frac{1}{p}}.
\end{aligned}\end{equation}
For the derivative factor we get, using again that $\dist_2$ is $1$-Lipschitz combined with the chain rule for Lipschitz and Sobolev functions \cite[Thm.~2.1.11]{Zie89} that
\begin{equation}\label{eq:distest4}\begin{aligned}
\|\nabla ( \dist_2 \circ \phi  - \dist_1 ) \|_{L^p(\Sigma)} &= \big\|( \nabla\dist_2 \circ \phi)^\T D\phi  - \nabla\dist_1\big\|_{L^p(\Sigma)}
\\&\ls \left(\|D\phi\|_{L^p(\Omega)}+|\Omega|^{\frac{1}{p}}\right) \ls C\left((\sigma^{-\theta}E_{\vol}^\sigma[\phi])^{\frac{1}{p}}+1\right) \\
& \ls C\left(\sigma^{-\frac{\theta}{p}}C_0^{\frac{1}{p}}+1\right) \ls C\left(\sigma^{-\frac{\theta}{p}}+1\right) \ls C\sigma^{-\frac{\theta}{p}}.
\end{aligned}\end{equation}
Combining \eqref{eq:distest3} and \eqref{eq:distest4} into \eqref{eq:interpolation}, and noticing that since $\theta \in \{0,1\}$ the second term of its right hand side is dominated by the first as $\sigma \to 0$, we get that
\begin{equation}\label{eq:distest5}
\| \dist_2 \circ \phi  - \dist_1 \|_{L^\infty(\Sigma)} \ls C \sigma^{-\frac{d\theta}{p^2}} \left( \sigma^{\frac{q}{p}}\big[\eta_\sigma(\sigma - \hat{\sigma})\big]^{-\frac{1}{p}} \right)^{1-\frac{d}{p}}.
\end{equation}
Now, if $\theta = 0$ we could just choose for example $\hat{\sigma} = \sigma/2$, so that $\eta_\sigma(\sigma/2)=\sigma^{-1}\eta(1/2)$ and \eqref{eq:distest5} becomes
\begin{equation}\label{eq:distest6}
\| \dist_2 \circ \phi  - \dist_1 \|_{L^\infty(\Sigma)} \ls C \sigma^{\left(\frac{q+1}{p}\right)\left(1-\frac{d}{p}\right)},
\end{equation}
and since this exponent is positive in particular for any for any $q>0$, we obtain the desired estimate.

In the case $\theta = 1$, the decay of $\eta$ needs to be taken into account, since we saw that to control the right hand side of \eqref{eq:distest2} the cutoff width $\hat{\sigma}$ needs to be closer and closer to $0$. With $\hat{\sigma}=\sigma^r$ as discussed above, and $\eta$ a spline of order $n$ we have $\eta_\sigma(\sigma - \sigma^r) = \sigma^{-1} \eta(1 - \sigma^{r-1})\gs C \sigma^{n(r-1)-1}$. Gathering exponents in \eqref{eq:distest5} this leads to the condition 
\begin{equation}\label{eq:distest7}\begin{gathered}
-\frac{d}{p^2}+\left(1-\frac{d}{p}\right)\left(\frac{q}{p}-\frac{n(r-1)-1}{p}\right)>0, \text{ or }\\
q > n(r-1)+\frac{d}{p-d}-1,
\end{gathered}\end{equation}
which for $r>\max(1/(p-d),1)$ is precisely \eqref{eq:coupledExponents}. 
\end{proof}

\begin{rem}\label{rem:magicalexponents}
We notice that if $p = d+1$ as chosen for the density $W$ in \eqref{eq:Wbounded}, any exponent $r>1$ can be chosen in the proof, and in turn condition \eqref{eq:coupledExponents} is independent of $n$ and simplifies to $q > d - 1$. However, for the above argument to remain valid $\sigma$ should still have polynomial decay and not faster, since otherwise we would have to replace the factor $\sigma^{-q}$ by a function increasing faster as well.
\end{rem}

\begin{thm}\label{thm:existence}Assume either \eqref{eq:decoupledExponents} or \eqref{eq:coupledExponents} and \begin{equation}\label{eq:sigmari}\sigma \in (0,\min\big(r_I,\sigma_{r_I})\big),\end{equation} where $r_I$ and $\sigma_{r_I}$ are defined as in the statement of Lemma \ref{lem:distest}. Then there exists at least one minimizer of $E^\sigma$ in $\mathcal{B}_0$.
\end{thm}
\begin{proof}
Let $\{\phi_k\}_k$ be a minimizing sequence. Using the boundary conditions and Poincar\'e inequality \cite[Eq.~(7.44)]{GilTru01} the term $E_{\vol}$ provides coercivity in $W^{1,p}(\Omega;\R^d)$, and by the Banach-Alaoglu theorem we can assume that this sequence weakly converges to some $\phi$ in $W^{1,p}(\Omega;\R^d)$. We denote $\psi_k=\left(\phi_k\right)^{-1}$ the corresponding inverses, whose existence is guaranteed by $\phi_k \in \mathcal{B}_0$. Possibly by taking another subsequence we can also assume that there is $\psi$ for which $\psi_k \wkto \psi$ weakly in $W^{1,p}(\Omega;\R^d)$, since $E_{\vol}[\phi_k] \gs \|D\psi_k\|_{L^p(\Omega)}$ as well. Now, on the one hand we can apply Ball's global invertibility theorem \cite[Thm.~2]{Bal81} and weak lower semicontinuity of $E_{\vol}$ to obtain that $\phi$ is a homeomorphism from $\Omega$ to $\Omega$ and $\phi^{-1} \in W^{1,p}(\Omega;\R^d)$, that is $\phi \in \mathcal{B}_0$. On the other, since $p>d$ the functions are uniformly continuous with modulus of continuity uniform in $k$, by coercivity in $W^{1,p}(\Omega; \R^d)$ and the Morrey inequality. Therefore by the Arzel\`{a}-Ascoli theorem, possibly by taking another subsequence we have that the convergence is also uniform, which allows us to conclude that the limit of inverses is the inverse of the limit, that is, $\psi=\phi^{-1}$.

We then note that the $\dist_i$ are $C^2$ in $\N_{r_I} \M_i$. To see this, since $\M_1$ is assumed to be $C^2$ we can apply \cite[Lem.~14.16]{GilTru01} or the results of \cite{Foo84} for the unsigned distance function on $\N_{r_I}\M_i \setminus \M_i$, and notice that the signed distance function $\dist_i$ also inherits this regularity \cite[Thm.~7.8.2(iii)]{DelZol11} in a neighborhood of each point of $\M_i$, which is compact. We can then apply \eqref{eq:sigmari} and Lemma \ref{lem:distest} to obtain that for $k$ large enough we have \[\phi_k\left( \N_\sigma \M_1 \right) \subset \N_{r_I} \M_2 \text{, and } \phi_k^{-1}\left( \N_\sigma \M_2 \right) \subset \N_{r_I} \M_1,\]
which implies that at values attained by $\phi_k$, the integrands \eqref{eq:Fmemdirect} and \eqref{eq:Fmeminverse} are continuous in their last two arguments. The same conclusion holds true for \eqref{eq:Fbenddirect} and \eqref{eq:Fbendinverse} after using a continuity result for square roots of nonnegative definite matrix-valued functions \cite[Thm.~1.1]{ChHu97} to account for the presence of $\mathcal{R}$ in $\S_i$. Lower semicontinuity of $E_{\mem}^\sigma, E_{\bend}^\sigma$ along $\phi_k$ then follows by Lemma \ref{lem:polyconvexity} and a lower semicontinuity theorem for integral functionals with Carath\'{e}odory energy densities which are polyconvex in their derivative argument \cite[Theorem 8.16]{Dac08}. We conclude that $\phi$ is the desired minimizer.
\end{proof}

\section{Scaling limits for non-symmetric energies}\label{sec:limits}
We now turn our attention to the limit of level set matching energies as the parameter $\sigma$ controlling the size of the narrow band goes to zero. When the symmetric energies $E^\sigma$ of \eqref{eq:Esym} are used, one should work in classes of invertible functions, which strongly limit the types of analysis possible (see Remark \ref{rem:whynonsym} below). Therefore, in this section, we only penalize the direct transformation and limit ourselves to the ``non-symmetric'' family of functionals $\mathcal{E}^\sigma:W^{1,p}_0(\Omega;\R^d)+\Id \to \R^+ \cup \{0\}$
\begin{equation}\label{eq:Edir}
\mathcal{E}^\sigma = \mathcal{E}_{\match}^\sigma+\mathcal{E}_{\mem}^\sigma+\mathcal{E}_{\bend}^\sigma+\mathcal{E}_{\vol}^\sigma,
\end{equation}
in which the contributions of the inverse deformation are not considered, so that 
\begin{equation}\label{eq:EdirMatch}
\mathcal{E}_{\match}^\sigma[\phi]:=\frac{1}{\sigma^q}\int_\Omega \eta_\sigma(\dist_1) |\dist_2 \circ \phi - \dist_1|^2 \dd x,
\end{equation}
\begin{equation}\label{eq:EdirMem}
\mathcal{E}_{\mem}^\sigma[\phi]:=\int_\Omega \eta_\sigma(\dist_1) W\big( (\P_2 \circ \phi) D\phi\, \P_1  + (\n_2 \circ \phi) \otimes \n_1 \big) \dd x,
\end{equation}
\begin{equation}\label{eq:EdirBend}
\mathcal{E}_{\bend}^\sigma[\phi]:=\int_\Omega \eta_\sigma(\dist_1) W\big( \Lambda[\S_1, \S_2 \circ \phi, D\phi, \n_1, \n_2 \circ \phi] \big) \dd x \text{, and }
\end{equation}
\begin{equation}\label{eq:EdirVol}
\mathcal{E}_{\vol}^\sigma[\phi]:=\sigma^\theta\int_\Omega W(D\phi) \dd x.
\end{equation}
We assume that $W\in C^1(\R^{d\times d})$ is such that for all $A, B \in \R^{d \times d}$
\begin{align}
W(A) &\gs C|A|^p - \frac{1}{C}, \text{ and }\label{eq:Wcoer}\\
W(A) &\ls C \big(|A|^p+1\big).\label{eq:Wbound}
\end{align}
These conditions are in particular satisfied by the density \eqref{eq:Wbounded}. The bound \eqref{eq:Wbound} combined with quasiconvexity implies (see \cite[p.~6]{Mar85} or \cite[Lem.~6.6]{Led18}) the continuity property 
\begin{equation}|W(A)-W(B)| \ls C |A-B|\big(1+|A|^{p-1}+|B|^{p-1}\big)\label{eq:Wcont}.\end{equation}
Alternatively, one can also check \eqref{eq:Wcont} for \eqref{eq:Wbounded} directly. For that, just recall \cite[Thm.~4.7]{Fri82} the inequality $|\det A- \det B|\ls C|A-B|\max(|A|,|B|)^{d-1}$ and notice that the function $t\mapsto \sqrt{1+(t-2)^2}$ has bounded derivative.

Since $\eta_\sigma$ has constant integral, the energy scaling of $\mathcal{E}^\sigma_\mem$ is the one of the classical membrane limit \cite{LedRao95, LedRao96}, whose results we apply directly. The structure of the proof is based on the methods delineated in \cite{AceButPer88, BesKraMic09}, where problems for thin inclusions or `welding' are considered. In particular, we will use the following lemma for integration by parts of non-intrinsic products on a hypersurface, analogous to  \cite[Prop.~II.2]{AceButPer88}:

\begin{lem}\label{lem:intbyparts}Let $N \in W^{1,p}(\Omega;\R^{d \times d})$ and $v \in W^{1,p}(\Omega;\R^d)$. Then for the traces of $v$ and $N$ on $\M_1$ the following are well defined and equal:
\begin{equation}\label{eq:intbyparts}\int_{\M_1}N : \Dt v \, \dd \mathcal{H}^{d-1} = - \sum_{i=1}^d \int_{\M_1} \Div_{\M_1}\!\big( [N \P_1]_i\big) \,v^i \, \dd \mathcal{H}^{d-1},\end{equation}
where $D_t v := Dv \P_1$ is the tangential derivative of $v$ on $\M_1$, $[N\P_1]_i$ is the $i$-th row of $N\P_1$, $v^i$ the $i$-th component of $v$, and $\Div_{\M_1}\!( [N \P_1]_i )$ is the Riemannian divergence on $\M_1$ applied to the tangential vector field $[N \P_1]_i$. 
\end{lem}
\begin{proof}We first assume that $N \in C^1(\M_1;\R^{d \times d})$ and $v \in C^1(\M_1;\R^d)$ to check \eqref{eq:intbyparts}. Since $\P_1$ is symmetric, $\P_1^2=\P_1$ and since the matrix trace is invariant under cyclic permutations, we have
\[\begin{aligned}N:\Dt v &= \tr\big( N^\T Dv \P_1 \big) = \tr\big( \P_1 N^\T Dv \big) = tr\big( \P_1^2 N^\T Dv \big) \\
&= \tr\big( \P_1 N^\T Dv \P_1 \big) = \tr\big( (N \P_1)^\T Dv \P_1 \big) = N \P_1 : \Dt v,\end{aligned}\]
Finally, since the metric on $\M_1$ is induced by its immersion into $\R^d$ and, being compact, it has no boundary, using the divergence theorem on $\M_1$ (see \cite[Section III.7]{Cha06}, for example) we get
\[\int_{\M_1}N : \Dt v = \int_{\M_1}N \P_1 : \Dt v = \sum_{i=1}^d \int_{\M_1} [N \P_1]_i \cdot \Dt v^i =-\sum_{i=1}^d \int_{\M_1} \Div_{\M_1}\!\big( [N \P_1]_i\big)\,v^i,\]
as claimed. Now, if $v \in W^{1,p}(\Omega;\R^d), M \in W^{1,p}(\Omega;\R^{d \times d})$ the traces of $v$ and $N$ on $\M_1$ are \cite[Prop.~3.31]{DemDem12} in $W^{1-\frac{1}{p}, p}(\M_1;\R^d)$ and $W^{1-\frac{1}{p}, p}(\M_1;\R^{d \times d})$, respectively. Since $\P_1 \in C^1(\M_1; \R^{d\times d})$ and $\M_1$ is $C^2$, the formula \eqref{eq:intbyparts} will also hold if both sides are well defined. This follows by the embedding (see \cite[Thm.~3.54]{DemDem12} for the dual space)
\[W^{-\frac{1}{p}, p}(\M_1) \subset \left(W^{1-\frac{1}{p}, p}(\M_1)\right)' = W^{\frac{1}{p'}-1, p'}(\M_1),\]
which holds because $1/p'-1 = -1/p$ and since $p>d\gs2$ we have $p' < p$, while $\M_1$ is compact.
\end{proof}

We are now ready to state and prove our convergence result. For convenience we denote for $x \in \Omega$ the tangential-projected derivative as
\begin{gather}\label{eq:Dttdef}\Dtt \phi(x) := \P_2(\phi(x)) D\phi(x)\, \P_1(x)  + \n_2(\phi(x)) \otimes \n_1(x),\text{ so that }\\
\mathcal{E}_{\mem}^\sigma[\phi]:=\int_\Omega \eta_\sigma(\dist_1) W(\Dtt\phi).
\end{gather}
Our main point is that this definition allows us to recover a surface functional with the same structure in the limit (compare the integrands in \eqref{eq:limitfunctional} and \eqref{eq:EdirMem}), which is typically not the case in dimension reduction problems. For the membrane problem in \cite{LedRao95, LedRao96} a quasiconvex envelope appears in the limit problem, which turns out to be trivial in our case.

\begin{rem}\label{rem:nosequences}
Since $L^p(\Omega; \R^d)$ is a metric space, $\Gamma$-convergence can be characterized \cite[Def.~1.5]{Bra02} in terms of the $\liminf$ and $\limsup$ inequalities. To simplify the notation we will continue to write the continuous parameter $\sigma \to 0$ while speaking of sequences. Strictly, what is implied is $\Gamma$-convergence of $\mathcal{E}^{\sigma_j}$ for any sequence $\{\sigma_j\}_{j \in \mathbb{N}}$ with $\sigma_j \to 0$. Likewise, when we speak of subsequences of $\phi_\sigma$, which are not relabelled, we mean sequences $\phi_{\sigma_j}$ for some sequence $\{\sigma_j\}_j$.
\end{rem}

\begin{thm}\label{thm:gammaconv}Let $W$ be polyconvex and satisfy \eqref{eq:Wcoer}, \eqref{eq:Wbound} and \eqref{eq:Wcont}. Define the set
\[\mathcal{T}_p:=\left\{\phi \in W^{1,p}_0(\Omega;\R^d)+\Id \left|\,\;\restr{\Dtt\phi}{\M_1}\in L^p(\M_1; \R^{d \times d}),\;\phi(\M_1)=\M_2 \right. \right\}.\]
Then assuming $\theta = 0$ and $q>0$, the family $\mathcal{E}^\sigma$ $\Gamma$-converges in the $L^p(\Omega; \R^d)$ topology as $\sigma \to 0$ to the functional $\mathcal{E}^0$ defined for $\phi \in \mathcal{T}_p$ by
\begin{equation}\label{eq:limitfunctional}\mathcal{E}^0[\phi]:=\int_{\M_1}W( \Dtt\phi )+W\big( \Lambda[D\phi, \S_1, \S_2 \circ \phi, \n_1, \n_2 \circ \phi] \big)\dd \mathcal{H}^{d-1} + \int_\Omega W(D\phi) \dd x\end{equation}
and $\mathcal{E}^0[\phi]=+\infty$ if $\phi \notin \mathcal{T}_p$. Moreover $\mathcal{E}^0$ possesses at least one minimizer in $W^{1,p}_0(\Omega;\R^d)+\Id$.
\end{thm}
\begin{proof}
Throughout the proof, to simplify notation we will not consider the bending-like term $\mathcal{E}^\sigma_{\bend}$. Since it consists on a pre- and post-stretched modification of $\mathcal{E}^\sigma_{\mem}$ where curvature-dependent coefficients are introduced, the proof for $\mathcal{E}^\sigma_{\mem}$ (which already contains varying coefficients depending on the deformed configuration) applies with completely straightforward modifications.

\medskip
\noindent \emph{Step 1: Energy bounds on a sequence imply tangential regularity of its limit.} 
\medskip

Let $\phi_\sigma \to \phi$ in $L^p(\Omega;\R^d)$ and assume that the sequence $\{\mathcal{E}^\sigma[\phi_\sigma]\}_\sigma$ is bounded as $\sigma \to 0$. Since $\theta = 0$, taking into account \eqref{eq:Wcoer} and that we work in $W^{1,p}_0(\Omega; \R^d)+\Id$, using the Poincar\'e inequality we have $\|\phi_\sigma\|^p_{W^{1,p}(\Omega)} \ls C(\mathcal{E}^\sigma[\phi_\sigma]+1) \ls C$, so that upon taking a subsequence we have $\phi_\sigma \wkto \phi$ weakly in $W^{1,p}(\Omega;\R^d)$ and also converging uniformly.

At first glance, the trace of $\phi$ on $\M_1$ is only in $W^{1-\frac{1}{p}, p}(\M_1;\R^d)$. However, as in \cite[Lem.~III.1]{AceButPer88} boundedness of the energies along the sequence $\phi_\sigma$ implies additional regularity for the trace and $\restr{\Dtt\phi}{\M_1} \in L^p(\M_1;\R^d)$. For this, we would like to exploit the bound 
\begin{equation}\label{eq:Emembound1}\mathcal{E}^\sigma_{\mem}[\phi_\sigma]=\int_\Omega \eta_\sigma (\dist_1 ) W(\Dtt \phi_\sigma) \ls C.\end{equation}
Our first step is to notice that an estimate for $\|\dist_2 \circ \phi_\sigma\|_{L^\infty(\N_\sigma \M_1)}$ analogous to that of Lemma \ref{lem:distest} also holds here. The main difference is that our proof of Lemma \ref{lem:distest} assumed that the deformations under consideration map $\Omega$ to $\Omega$, but now this is not guaranteed since $\mathcal{E}_{\vol}^\sigma$ contains no injectivity penalization. This difficulty can be overcome by modifying \eqref{eq:distest3} with the estimate, obtained using $\restr{\phi_\sigma}{\partial \Omega} = \Id$ and that $\dist_1,\dist_2$ are $1$-Lipschitz,
\[\begin{aligned}\|\dist_2 \circ \phi_\sigma- \dist_1 \|_{L^\infty( \Sigma )}&\ls \sup_{x \in \partial \Omega}(\dist_2 \circ \phi_\sigma)(x) + |\phi_\sigma|_{C^{0,\alpha}(\Omega)} (\diam \Omega )^\alpha + \diam \Omega \\&\ls |\phi_\sigma|_{C^{0,\alpha}(\Omega)} (\diam \Omega )^\alpha + 2\diam \Omega ,\end{aligned}\]
on which again one can use the Morrey inequality and energy bounds. This modification affects the exponents appearing in \eqref{eq:distest5}, but only by terms proportional to $\theta$, which in this case is zero. Therefore, for some $\sigma_0$ small enough and since $\mathcal{E}^\sigma[\phi_\sigma] \ls C$, we have that
\begin{equation}\label{eq:contgeom}\dist_1 \in C^2(\N_\sigma \M_1), \dist_2 \in C^2\left(\phi\big(\N_\sigma \M_1\big)\right) \text{ and } \dist_2 \in C^2\left(\phi_\sigma\big(\N_\sigma \M_1\big)\right) \text{ for all } \sigma \in (0, \sigma_0).\end{equation}
To simplify the computations that follow, we first replace the coefficients appearing in $\mathcal{E}^\sigma_{\mem}[\phi_\sigma]$ that depend on $\phi_\sigma$ by those corresponding to the limiting function $\phi$. Using \eqref{eq:contgeom} so that $\P_1,\P_2, \n_1,\n_2$ are uniformly continuous where they are evaluated, and the continuity hypothesis \eqref{eq:Wcont} for the matrix fields $A_\sigma = (\P_2 \circ \phi_\sigma) D\phi_\sigma \P_1 + (\n_2 \circ \phi_\sigma) \otimes \n_1$ and $A := ( \P_2 \circ \phi) D\phi_\sigma \P_1 + (\n_2 \circ \phi) \otimes \n_1$ we obtain
\[|W(A_\sigma)-W(A)|\ls C (|D\phi_\sigma|+1)|\phi_\sigma - \phi|\left(1+|D\phi|^{p-1}+|D\phi_\sigma|^{p-1}\right).\]
Integrating and using the H\"older inequality, we see that the error we commit in the energy can be bounded by
\begin{equation}\label{eq:approxBound}C\|\phi_\sigma - \phi\|_{L^\infty(\Omega)}\left(1+\|D\phi_\sigma\|_{L^p(\Omega)}\|D\phi\|_{L^p(\Omega)}^{p-1}+\|D\phi_\sigma\|^p_{L^p(\Omega)}\right),\end{equation}
which clearly tends to zero as $\sigma \to 0$. With these fixed coefficients we denote
\[\overline\Dtt \phi_\sigma := (\P_2\circ \phi) D\phi_\sigma\, \P_1  + (\n_2 \circ \phi) \otimes \n_1.\]
Using the boundedness of $W$ in \eqref{eq:Wbound}, the Tonelli theorem to slice along offset hypersurfaces, and the change of variables $y \to y +t\n_1(y)$, the bounds \eqref{eq:Emembound1} and \eqref{eq:approxBound} mean that 
\begin{equation}\label{eq:Emembound2}\begin{aligned}
&\int_\Omega \eta_\sigma(\dist_1) |\overline\Dtt\phi_{\sigma}(x)|^p \dd x \\
&\enskip= \frac{1}{\sigma} \int_{\supp{\eta_\sigma(\dist_1)}} \eta\left(\frac{\dist_1}{\sigma} \right) |\overline\Dtt\phi_{\sigma}(x)|^p \dd x\\ 
&\enskip=\int_{-\sigma}^{\sigma} \int_{\M_1}\eta_\sigma(t) \left|\overline\Dtt\phi_{\sigma}\big(y + t \n_1(y)\big)\right|^p \left|\det\big(\IdM + t D \n_1(y)\big)\right| \dd \mathcal{H}^{d-1}(y) \dd t \\
&\enskip=\int_{\M_1} \int_{-\sigma}^{\sigma} \eta_\sigma(t) \left|\overline\Dtt\phi_{\sigma}\big(y + t \n_1(y)\big)\right|^p \left|\det\big(\IdM + t D \n_1(y)\big)\right| \dd t \dd \mathcal{H}^{d-1}(y)\ls C.
\end{aligned}\end{equation}
Using \eqref{eq:contgeom} we have that 
\begin{equation}\label{eq:nocutlocus}\det\big(\IdM + t D \n_1(y)\big) = \det\left(\IdM + t D^2 \dist_1(y)\right) > c > 0\text{ for all }y \in \M_1\text{ and }|t|<\sigma,\end{equation}
which combined with \eqref{eq:Emembound2} implies
\begin{equation}\label{eq:Emembound3}\int_{\M_1} \int_{-\sigma}^{\sigma} \eta_\sigma(t) \left|\overline\Dtt\phi_{\sigma}\big(y + t \n_1(y)\big)\right|^p \dd t \dd \mathcal{H}^{d-1}(y) \ls C.\end{equation}
On the other hand, observing that \[\int_{-\sigma}^\sigma \eta_\sigma(t) \dd t = \int_{-1}^1 \eta(t) \dd t = 1,\] we can use Jensen's inequality for the measure $\eta_\sigma(t) \dd t$ and \eqref{eq:Emembound3} to obtain
\begin{align*}&\int_{\M_1} \left| \frac{1}{\sigma} \int_{-\sigma}^{\sigma} \eta\left(\frac{t}{\sigma}\right) \overline\Dtt\phi_{\sigma}\big(y + t \n_1(y)\big) \dd t \right|^p \dd \mathcal{H}^{d-1}(y)\\ 
&=\int_{\M_1} \left| \int_{-\sigma}^{\sigma} \eta_\sigma(t) \overline\Dtt\phi_{\sigma}\big(y + t \n_1(y)\big) \dd t \right|^p \dd \mathcal{H}^{d-1}(y)\\
&\ls \int_{\M_1} \int_{-\sigma}^{\sigma} \eta_\sigma(t) \left| \overline\Dtt\phi_{\sigma}\big(y + t \n_1(y)\big) \right|^p \dd t \dd \mathcal{H}^{d-1}(y) \ls C.
\end{align*}
Therefore, the sequence $u_\sigma \in L^p(\M_1; \R^{d \times d})$ of tangential derivatives averaged along normals given by 
\[u_\sigma(y) := \frac{1}{\sigma} \int_{-\sigma}^{\sigma} \eta\left(\frac{t}{\sigma}\right) \overline\Dtt\phi_{\sigma}\big(y + t \n_1(y)\big) \dd t\]
can be assumed, upon possibly taking another subsequence, to converge weakly to some limit in $L^p(\M_1; \R^{d \times d})$. To identify the limit, by density we may test this weak convergence with $F \in C^1(\M_1; \R^{d \times d})$. Using \eqref{eq:contgeom}, that $\phi$ is uniformly continuous and that $\phi_\sigma$ is bounded in $W^{1,p}(\Omega; \R^d)$, we obtain for some functions $h_1, h_2$ with $h_j(\sigma) \to 0$ as $\sigma \to 0$ that
\begin{align*}
& \int_{\M_1}\!F(y) : \left( \overline\Dtt \left[ \int_{-\sigma}^{\sigma} \eta_\sigma(t) \phi_{\sigma}\big(\cdot + t \n_1(\cdot)\big) \dd t \right]\!(y) -(\n_2 \circ \phi)(y) \otimes \n_1(y)\right) \dd \mathcal{H}^{d-1}(y)\\
&=\int_{\M_1}\!F(y) :(\P_2\circ \phi)(y)\, D\! \left[ \int_{-\sigma}^{\sigma} \eta_\sigma(t) \phi_{\sigma}\big(\cdot + t \n_1(\cdot)\big) \dd t \right]\!(y) \,\P_1(y) \dd \mathcal{H}^{d-1}(y)\\
& =\int_{\M_1}\!F(y) : \int_{-\sigma}^{\sigma} \eta_\sigma(t) (\P_2\circ \phi)(y)\Big[ D\phi_{\sigma}\big(y + t \n_1(y)\big)P_1(y)+tD\phi_\sigma\big(y+t\n_1(y)\big)D\n_1(y)\Big] \dd t \dd \mathcal{H}^{d-1}(y)\\
&=\int_{\M_1}\!F(y) : \int_{-\sigma}^{\sigma} \eta_\sigma(t) (\P_2\circ \phi)\big(y + t \n_1(y)\big)D\phi_{\sigma}\big(y + t \n_1(y)\big)\P_1\big(y + t \n_1(y)\big) \dd t \dd \mathcal{H}^{d-1}(y) + h_1(\sigma)\\
&=\int_{\M_1}\!F(y) :\left( \int_{-\sigma}^{\sigma} \eta_\sigma(t) \overline\Dtt\phi_{\sigma}\big(y + t \n_1(y)\big) \dd t -(\n_2 \circ \phi)(y) \otimes \n_1(y)\right) \dd \mathcal{H}^{d-1}(y)+h_2(\sigma),
\end{align*}
where the additional error $h_2-h_1$ accounts for the difference in the last term 
\[(\n_2 \circ \phi)\big(y+t\n_1(y)\big) \otimes \n_1\big(y+t\n_1(y)\big)-(\n_2 \circ \phi)(y) \otimes \n_1(y).\]
Noticing that $F:(\n_2 \otimes \n_1)=\n_2^\T F \n_1$, the above computation, $\int \eta_\sigma = 1$, that $\P_2^\T = \P_2$, and integrating by parts on $\M_1$ with Lemma \ref{lem:intbyparts} we get
\begin{align*}
&\int_{\M_1}\! F(y) \mkern1mu:\mkern1mu u_\sigma(y) \dd \mathcal{H}^{d-1}(y) -\int_{\M_1}\!(\n_2 \circ\phi)^\T(y) \,F(y)\, \n_1(y) \dd \mathcal{H}^{d-1}(y)\\
&\quad=  \int_{\M_1}\! F(y) : \left(\int_{-\sigma}^{\sigma} \eta_\sigma(t) \overline\Dtt\phi_{\sigma}\big(y + t \n_1(y)\big)\dd t-(\n_2 \circ \phi)(y) \otimes \n_1(y)\right) \dd \mathcal{H}^{d-1}(y)\\
&\quad=  \int_{\M_1}\! F(y) : \left(\overline\Dtt\left[ \int_{-\sigma}^{\sigma} \eta_\sigma(t) \phi_{\sigma}\big(\cdot + t \n_1(\cdot)\big) \dd t \right]\!(y) - (\n_2 \circ \phi)(y) \otimes \n_1(y) \right)\dd \mathcal{H}^{d-1}(y) - h_2(\sigma) \\
&\quad=  \int_{\M_1}\! (\P_2\circ \phi)(y) F(y) \mkern1mu:\mkern1mu \Dt\left[ \int_{-\sigma}^{\sigma} \eta_\sigma(t) \phi_{\sigma}\big(\cdot + t \n_1(\cdot)\big) \dd t \right]\!(y) \dd \mathcal{H}^{d-1}(y) - h_2(\sigma)\\
&\quad= - \sum_{i=1}^d \int_{\M_1}\! \Div_{\M_1}\!\Big(\big[(\P_2 \circ \phi)(y)F(y)\P_1(y)\big]_i\Big) \left[\int_{-\sigma}^{\sigma} \eta_\sigma(t) \phi_{\sigma}^i\big(y + t \n_1(y)\big) \dd t \right] \dd \mathcal{H}^{d-1}(y) - h_2(\sigma).
\end{align*}
Now, using the weak convergence $\phi_\sigma \wkto \phi$ in $W^{1,p}(\Omega; \R^d)$ combined with weak continuity \cite[Ex.~3.2]{DemDem12} of the trace map from $W^{1,p}(\Omega)$ onto $W^{1-\frac{1}{p}, p}(\M_1)$ we get that
\[\int_{-\sigma}^{\sigma} \eta_\sigma(t) \phi_{\sigma}^i\big(\cdot + t \n_1(\cdot)\big) \dd t \xrightharpoonup[\sigma\to 0]{} \phi^i(\cdot) \text{ in }W^{1-\frac{1}{p}, p}(\M_1),\]
so that using Lemma \ref{lem:intbyparts} again we end up with
\begin{align*}
&- \sum_{i=1}^d \int_{\M_1}\! \Div_{\M_1}\!\Big(\big[(\P_2 \circ \phi)(y)F(y)\P_1(y)\big]_i\Big) \left[\int_{-\sigma}^{\sigma} \eta_\sigma(t) \phi_{\sigma}^i\big(y + t \n_1(y)\big) \dd t \right] \dd \mathcal{H}^{d-1}(y) - h_2(\sigma)\\
&\quad\xrightarrow[\sigma \to 0]{} - \sum_{i=1}^d \int_{\M_1}\! \Div_{\M_1}\!\Big(\big[(\P_2 \circ \phi)(y)F(y)\P_1(y)\big]_i\Big)\;\phi^i(y) \dd \mathcal{H}^{d-1}(y)\\
&\quad= \int_{\M_1}\! \Big( \P_2 \circ \phi)(y) F(y) \Big)\,:\, \Dt\phi(y) \,\dd \mathcal{H}^{d-1}(y)\\
&\quad= \int_{\M_1}\! F(y) \mkern1mu:\mkern1mu \Big( \Dtt \phi(y) - (\n_2 \circ\phi)(y) \otimes \n_1(y)\Big) \dd \mathcal{H}^{d-1}(y)\\
&\quad= \int_{\M_1}\! F(y) \mkern1mu:\mkern1mu \Dtt \phi(y) \dd \mathcal{H}^{d-1}(y) - \int_{\M_1}(\n_2 \circ\phi)^\T(y) \,F(y)\, \n_1(y) \dd \mathcal{H}^{d-1}(y),
\end{align*}
whence we identify the weak limit of $u_\sigma$ and deduce that $\Dtt \phi \in L^p(\M_1)$, and therefore $\phi \in \mathcal{T}_p$.

Note that this step implies that the $\Gamma$-limit of $\mathcal{E}^\sigma$ equals $+\infty$ whenever $\phi \notin \mathcal{T}_p$: if we had $\Dtt \phi \notin L^p(\M_1)$, having any sequence $\phi_\sigma$ with $\phi_\sigma \wkto \phi$ in $W^{1,p}(\Omega;\R^d)$ and $\liminf_{\sigma \to 0}\mathcal{E}^\sigma[\phi_\sigma] < +\infty$ would be a contradiction with the above.

\bigskip
\noindent \emph{Step 2: $\liminf$ inequality.}
\medskip

We perform a localization procedure analogous to the one in \cite[Lem.~4.1]{IglRumSch18}, fixing the coefficients to those corresponding to the limit deformation, and then taking into account that all the functions involved in the coefficients are uniformly continuous. As remarked above, we do not take into account the bending-like energy $\mathcal{E}_\bend$, since the proof for it is completely analogous to that for the membrane energy $\mathcal{E}_\mem$.

Let $\phi \in \mathcal{T}_p$ and $\phi_\sigma \to \phi$ in $L^p(\Omega;\R^d)$. Possibly by taking a subsequence that does not alter $\liminf \mathcal{E}^\sigma[\phi_\sigma]$ we may assume that $\mathcal{E}^\sigma[\phi_\sigma] \ls C$, since otherwise there is nothing to prove. As in the previous step, using the coercivity of $\mathcal{E}^\sigma$ we may take another subsequence so that $\phi_\sigma \wkto \phi$ weakly in $W^{1,p}(\Omega;\R^d)$ and also uniformly. By definition $\mathcal{E}_\match[\phi_\sigma] \gs 0$, so clearly
\[0 \ls \liminf_{\sigma \to 0} \mathcal{E}_\match[\phi_\sigma].\]
For the volume term, it is enough to notice that $W$ is polyconvex and $\phi^\sigma \wkto \phi$ in $W^{1,p}(\Omega; \R^d)$, so by a standard lower semicontinuity theorem \cite[Theorem 8.16]{Dac08} we have
\[\mathcal{E}_{\vol}^\sigma[\phi] = \int_\Omega W(D\phi(x)) \dd x \ls \liminf_{\sigma \to 0} \int_\Omega W(D\phi_\sigma(x)) \dd x = \liminf_{\sigma \to 0} \mathcal{E}_{\vol}^\sigma[\phi_\sigma].\]

For the membrane term, as in the previous step we may assume \eqref{eq:contgeom} and replace $\n_2 \circ \phi_\sigma, \P_2 \circ \phi_\sigma$ by $\n_2 \circ \phi, \P_2 \circ \phi$ with vanishing error \eqref{eq:approxBound} in the energy. Next, we need to take care of the spatial dependency of the coefficients. To do this, we split $\N_\sigma \M_1$ in small subdomains on each of which the coefficients will be replaced with fixed ones. In this case we choose the subdomains to be of cylindrical shape (i.e. of constant height along a fixed vector), to then apply the results of \cite{LedRao95}. For this, given a small parameter $\delta >0$, define a collection of $N_\delta$ subsets $O_i^\delta \subset \M_1$, relatively open in $\M_1$ with
\[O_i^\delta \cap O_j^\delta = \emptyset,\;\text{diam}(O_i^\delta)<\delta, \text{ and }\M_1 \mathbin{\big\backslash} \bigcup_{i=1}^{N_\delta}O_i^\delta \text{ of zero }\mathcal{H}^{d-1}\text{ measure}.\]
We then choose for each $O_i^\delta$ a single point $x_i^\delta \in O_i^\delta$ such that $O_i^\delta$ may be written as a graph in direction $\n_1(x_i^\delta)$: since $\M_1$ is $C^2$, for small enough $\delta$ this is possible for all $i=1,\ldots,N_\delta$ simultaneously. We denote by $K_i^{\delta,\sigma}$ the neighborhood of width $2\sigma$ in the direction $\n_1(x_i^\delta)$ associated to each of the $O_i^\delta$, that is
\begin{equation}\label{eq:cylnbhd}K_i^{\delta,\sigma} := \left\{ y + t \n_1(x_i^\delta) \middle|\, y \in O_i^\delta,\;t\in\left(-\sigma, \sigma\right)\right\},\end{equation}
for which assuming $\sigma < \delta/2$ we have $\text{diam}(K_i^{\delta, \sigma}) \ls 2\delta$. 
We aim then to replace $\mathcal{E}_{\mem}[\phi_\sigma]$ by the sum over $i=1,\ldots,N_\delta$ of the integrals
\begin{equation}\label{eq:localizedEnergy}I_i^{\delta, \sigma}[\phi_\sigma] := \int_{K_i^{\delta,\sigma}}\eta_\sigma\big(t(x)\big)W\Big(\P_2\big(\phi(x_i^\delta)\big) D\phi_\sigma(x) \P_1(x_i^\delta) + \n_2\big(\phi(x_i^\delta)\big) \otimes n_1(x_i^\delta) \Big)\, \dd x,\end{equation}
where $t(x)$ is determined from \eqref{eq:cylnbhd}. The total error we commit when doing this replacement can be bounded by
\begin{equation}\label{eq:localizationError}C\big(\omega(\delta)^p+\omega(\delta)+\sigma\big).\end{equation}
Here, $\omega(\delta)$ is a modulus of continuity valid on $\mathcal{N}_\sigma \M_1$ for $\n_1, D^2\dist_1$ and the compositions of all the $\phi_\sigma$ with $\n_2$ and $D^2 \dist_2$, which exists because these converge uniformly and we have assumed \eqref{eq:contgeom}. The first term is derived using \eqref{eq:Wcont} analogously to \eqref{eq:approxBound} and reflects the error in the coefficients of $I_i^{\delta, \sigma}$. The second arises from the use of $\eta_\sigma(t(x))$ instead of $\eta_\sigma(\dist_1(x))$, since $\omega(\delta)$ is also a modulus of continuity for the curvature of $\M_1$. The third term accounts for the difference in the domains of integration and overlaps that arise because of the curvature of $O_i^\delta$, that is, the sets
\[\bigcup_{i \neq j} K_i^{\delta,\sigma} \cap K_j^{\delta,\sigma} \text{ and } \big(\N_\sigma \M_1\big) \Delta \bigg(\bigcup_{i}K_i^{\delta, \sigma}\bigg)\]
whose total measure is bounded by $C\sigma^2$, with effect magnified by a factor $\sigma^{-1}$ since $\eta_\sigma(\cdot) = \sigma^{-1} \eta(\cdot / \sigma)$. For all terms, \eqref{eq:Wbound} and the bound $\|D\phi_\sigma\|_{L^p(\Omega)}\ls C$ have been used.

Now for each $I_i^{\delta, \sigma}$, denoting by $Q_i(x)=Q(\n_i(x))$, defined as in Section \ref{subsec:notation} so that $Q_i(x)e_d=\n_i(x)$, and $\P(e_d)=\IdM - e_d \otimes e_d$ we notice that for any $A \in \R^{d \times d}$, we have by the symmetries of $W$ that
\begin{equation}\label{eq:rotateinside}\begin{aligned}&W\left( \P_2\big(\phi(x_i^\delta)\big) \,A\, \P_1(x_i^\delta) + \n_2\big(\phi(x_i^\delta)\big) \otimes \n_1(x_i^\delta) \right)\\
&= W\left( Q_2\big(\phi(x_i^\delta)\big)^\T \Big[\P_2\big(\phi(x_i^\delta)\big) \,A\, \P_1(x_i^\delta) + \n_2\big(\phi(x_i^\delta)\big) \otimes \n_1(x_i^\delta)\Big] Q_1(x_i^\delta)\right)\\
&= W\left(\P(e_d) \,A_{x_i^\delta}\, \P(e_d) + e_d \otimes e_d \right),\\
\end{aligned}\end{equation}
where $A_{x_i^\delta} := Q_2\big(\phi(x_i^\delta)\big) A\, Q_1(x_i^\delta)^T$.
Since $Q_2\big(\phi(x_i^\delta)\big)$ and $Q_1(x_i^\delta)^T$ are fixed matrices, they commute with differentiation, so that we can absorb the coordinate change and equivalently consider the sequence of deformations $x \mapsto Q_2\big(\phi(x_i^\delta)\big)\, \phi_\sigma\left(Q_1(x_i^\delta)^T x\right)$. 

After these transformations and since $\int \eta_\sigma = 1$ for all $\sigma$, we are in a position to apply the nonlinear membrane limit for plates of \cite[Thm.~2]{LedRao95}. Although $O_i^\delta$ is not flat, after fixing the coefficients and working in the cylindrical neighborhoods $K_i^{\delta, \sigma}$ defined in \eqref{eq:cylnbhd}, the constant vector $\n_1(x_i^\delta)$ plays the role of the vertical direction along which the rescaling of the membrane limit happens (see also the similar geometric situation considered in \cite[Sec.~2, Prop.~5]{BesKraMic09}). In this situation, one rescales $I_i^{\delta, \sigma}$ to the unit-height neighborhood
\begin{equation}\label{eq:stretchednbhd}\widehat{K}_i^\delta := \left\{ y + t \n_1(x_i^\delta) \middle|\, y \in O_i^\delta,\;t\in\left(-1,1\right)\right\} 
\text{ through }K_i^{\delta,\sigma} \ni y + t \n_1(x_i^\delta) \mapsto y + \frac{t}{\sigma} \n_1(x_i^\delta),\end{equation}
and notices that since $\|D\phi_\sigma\|_{L^p(\Omega)}\ls C$, the corresponding rescalings of $\restr{\phi_\sigma}{K_i^{\delta,\sigma}}$ satisfy
\begin{equation}\label{eq:stretchedconv}\begin{gathered}\widehat{\phi}_\sigma \in W^{1,p}\big(\widehat{K}_i^\delta; \R^d\big)\text{ and }\int_{\widehat{K}_i^\delta} \big|D\widehat{\phi}_\sigma(x) \,\n_1(x_i^\delta)\big|^p \dd x=\sigma^{p-1}\int_{K_i^{\delta, \sigma}} \big|D\phi_\sigma(x) \,\n_1(x_i^\delta)\big|^p \dd x, \text{ so that }\\D\widehat{\phi}_\sigma \,\n_1(x_i^\delta)\to 0 \text{ strongly in }L^p\big(\widehat{K}_i^\delta\big).\end{gathered}\end{equation}
The resulting $\Gamma$-limit has as integrand the transformation through \eqref{eq:rotateinside} of the quasiconvex envelope \cite[Sec.~6.1, Thm.~6.9]{Dac08} $QW^{x_i^\delta}$ of the density $W^{x_i^\delta}$ defined by
\begin{equation}\label{eq:membraneDensity}\begin{aligned}W^{x_i^\delta}(B):&=\inf_{\xi \in \R^3}W\left( \P(e_d) \big[ B^1 B^2 \ldots B^{d-1} \,\big|\, \xi\,\big] \P(e_d) + e_d \otimes e_d \right)\\
&=W\left( \P(e_d) \big[ B^1 B^2 \ldots B^{d-1} \,\big|\, 0\,\big] \P(e_d) + e_d \otimes e_d \right) = W\big( \P(e_d) B\, \P(e_d) + e_d \otimes e_d \big)\end{aligned}\end{equation}
applied at $B=A_{x_i^\delta}$. Here, $\big[ B^1 B^2 \ldots B^{d-1} \,\big|\, \xi\,\big]$ denotes the matrix obtained by replacing the last column of $B$ with $\xi$, and the infimum is trivial since the rightmost projection $\P(e_d)$ ensures that there is no dependence on $\xi$. The right hand side of \eqref{eq:membraneDensity} is polyconvex by Lemma \ref{lem:polyconvexity}, hence also quasiconvex \cite[Thm.~5.3]{Dac08} and therefore $QW^{x_i^\delta} = W^{x_i^\delta}$. In consequence, taking into account \eqref{eq:rotateinside} and \eqref{eq:membraneDensity} we have
\begin{equation}\label{eq:localizedLiminf}\widehat{I}_i^\alpha[\,\widehat{\phi}\,\big] := \int_{\widehat{K}_i^\delta} W\Big(\P_2\big(\phi(x_i^\delta)\big) D\widehat{\phi}(x) \P_1(x_i^\delta) + \n_2\big(\phi(x_i^\delta)\big) \otimes \n_1(x_i^\delta) \Big) \dd x \ls \liminf_{\sigma \to 0} I_i^{\delta, \sigma}[\phi_\sigma].\end{equation}
Here, the left hand side contains the extension $\widehat{\phi}$ of $\restr{\phi}{O_i^\delta}$ to $\widehat{K}_i^\delta$ defined by $\widehat{\phi}(x)=\phi(y)$ if $x\in \widehat{K}_i^\delta$ and $y$ is as in \eqref{eq:stretchednbhd}. 
Once again using \eqref{eq:Wcont} analogously to \eqref{eq:approxBound}, using $\sigma < \delta/2$, and since $\widehat{\phi}$ is constant in the direction $\n_1(x_i^\delta)$, we can estimate
\begin{equation}\label{eq:localizedFixingEst}\left| \widehat{I}_i^\alpha\big[\,\widehat{\phi}\,\big] - \int_{\widehat{K}_i^\delta} W\Big(\P_2\big(\widehat{\phi}(x)\big) D\widehat{\phi}(x) \P_1\big(y(x)\big) + \n_2\big(\widehat{\phi}(x)\big) \otimes \n_1\big(y(x)\big) \Big) \dd x \right| \ls C\omega(\delta)^p\mathcal{H}^{d-1}(O_i^\delta),\end{equation}
where $y(x)$ is the projection along $\n_1(x_i^\delta)$ onto $O_i^\delta$ as in the definition of $\widehat{K}_i^\delta$ in \eqref{eq:stretchednbhd}. Moreover, again because $\widehat{\phi}$ is constant in the direction $\n_1(x_i^\delta)$ and since $\phi \in \mathcal{T}_p$ we also have that
\[\int_{\widehat{K}_i^\delta} W\Big(\P_2\big(\widehat{\phi}(x)\big) D\widehat{\phi}(x) \P_1\big(y(x)\big) + \n_2\big(\widehat{\phi}(x)\big) \otimes \n_1\big(y(x)\big) \Big) \dd x = \int_{O_i^\delta} W\big(\Dtt \phi(y)\big)\dd \mathcal{H}^{d-1}(y),\]
so that summing over $i=1,\ldots,N_\delta$ and letting $\delta \to 0$, we conclude.
\bigskip

\noindent \emph{Step 3: $\limsup$ inequality.}
\medskip

Let $\phi \in \mathcal{T}_p$. We show that there exists a recovery sequence $\phi_{\sigma}$ for $\mathcal{E}^0$ at $\phi$, such that in addition we have
\begin{equation}\label{eq:perfmatch}\int_\Omega \eta_{\sigma} \circ \dist_1 |\dist_2\circ \phi_{\sigma} - \dist_1|^2=0.\end{equation}

Assume that $2\sigma < (\sup_{x \in \M^1}|D^2\dist_1(x)|)^{-1}$, so that as in the proof of Lemma \ref{lem:distest} each $x \in \mathcal N_{2\sigma} \M_1$ can be written uniquely as $x = y + t \n_1(y)$ with $y \in \M_1$, and denote the projection of x onto $\M_1$ by $\pi_{\M_1}(x):=y$. We then define the modified deformations $\phi_\sigma$ by 
\begin{equation}\label{eq:modphi}
\phi_\sigma(y+t\n_1(y)) = \tau_\sigma(t) \Big(\phi(y)+t\n_2\big(\phi(y)\big)-\phi\big(y+t\n_1(y)\big)\Big)+\phi\big(y+t\n_1(y)\big),\end{equation}
whenever $x \in \mathcal N_{2\sigma} \M_1$ and $\phi_\sigma(x)=\phi(x)$ otherwise. Here $\tau_\sigma: \R \to \R$ is nonincreasing and such that
\begin{equation}\label{eq:proptau}\tau_\sigma \gs 0,\;\tau_\sigma(t)=1 \text{ for } |t| \ls \sigma,\;\tau_\sigma(t)=0 \text{ for }|t| \gs 2\sigma, \text{ and } \left|\frac{\dd \tau_\sigma}{\dd t} \right|\ls \frac{2}{\sigma}.\end{equation}
Moreover, as done in the previous steps and using estimates analogous to those of Lemma \ref{lem:distest}, we consider only $\sigma$ is small enough for which $\dist_1 \in C^2(\N_{2\delta}\M_1)$ and $\dist_2 \in C^2\big(\phi(\N_{2\delta}\M_1)\big)$.

We aim then to show that $\phi_\sigma$ is a recovery sequence, that is
\[\limsup_{\sigma \to 0}\mathcal{E}^{\sigma}[\phi_\sigma]\ls\mathcal{E}^0[\phi].\]
First, notice that whenever $t \ls \sigma$ we have  
\[\phi_\sigma\big(y+t\n_1(y)\big) = \phi(y)+t\n_2\big(\phi(y)\big),\]
so \eqref{eq:perfmatch} is satisfied. Moreover, this also implies that 
\[\lim_{\sigma \to 0}\mathcal{E}_{\mem}^{\sigma}[\phi_\sigma]=\int_{\M_1} W\big(\Dtt\phi(y)\big) \dd \mathcal{H}^{d-1}(y).\]
To see this it suffices to notice, using the continuity hypothesis \eqref{eq:Wcont}, \eqref{eq:contgeom} and $\phi \in \mathcal{T}_p$, that 
\begin{align*}
&\left|\int_\Omega (\eta_{\sigma} \circ \dist_1) W\big(\Dtt\phi_\sigma(x)\big) \dd x - \int_{\M_1} W\big(\Dtt\phi(y)\big) \dd \mathcal{H}^{d-1}(y) \right| \\
&\enskip = \Bigg|\int_{-\sigma}^{\sigma} \eta_\sigma(t) \int_{\M_1}  W\left(\Dtt\phi_\sigma\big(y + t \n_1(y)\big)\right)\left|\det\left(\IdM + t D^2 \dist_1(y)\right)\right|\dd \mathcal{H}^{d-1}(y) \dd t \\
&\qquad - \int_{-\sigma}^{\sigma} \eta_\sigma(t) \dd t \int_{\M_1} W\big(\Dtt\phi(y)\big) \dd \mathcal{H}^{d-1}(y) \Bigg| \\
&\enskip \ls \sup_{z \in \M_1, |t|\ls \sigma}\bigg|1-\Big| \det\big( \IdM + tD^2 \dist_1(z)\big) \Big|\bigg|\left( \, \int_\Omega (\eta_{\sigma} \circ \dist_1) W\big(\Dtt\phi_\sigma(x)\big) \dd x \,\right)\\
&\qquad + \int_{-\sigma}^{\sigma} \eta_\sigma(t) \int_{\M_1} \Big|W\left(\Dtt\phi_\sigma\big(y + t \n_1(y)\big)\right) - W\big(\Dtt\phi(y)\big)\Big|\dd \mathcal{H}^{d-1}(y) \dd t \\
&\enskip \ls C\sigma^{d-1} + \int_{-\sigma}^{\sigma} \eta_\sigma(t) \int_{\M_1} \Big|W\left(\Dtt\big(\phi + t \n_2\circ\phi)(y)\right) - W\big(\Dtt\phi(y)\big)\Big|\dd \mathcal{H}^{d-1}(y) \dd t \\
&\enskip \ls C\sigma^{d-1} + C\int_{-\sigma}^{\sigma} t\,\eta_\sigma(t) \int_{\M_1} \big|\Dtt(\n_2\circ\phi)(y)\big|\big(1+|\Dtt\phi(y)|^{p-1}\big)\dd \mathcal{H}^{d-1}(y) \dd t \\
&\enskip \ls C\sigma^{d-1} + C\sigma^2\left(1+\|\Dtt\phi\|^{p-1}_{L^p(\M_1)}\right)\ls C(\sigma^{d-1} + \sigma^2)\xrightarrow[\sigma\to 0]{} 0.
\end{align*}

When considering the volume term, the transition layer in $\tau_\sigma$ between $\sigma$ and $2\sigma$ plays a role. We can estimate using the definition of $\phi_\sigma$ in \eqref{eq:modphi}, assumption \eqref{eq:Wbound}, that $\phi_\sigma \in W^{1,p}(\N_\sigma \M_1; \R^d)$ since in that subdomain it is the constant extension along the normal $\n_1$ of the trace $\restr{\phi}{\M_1}\in W^{1-\frac{1}{p}, p}(\M_1;\R^d)$, and the properties of $\tau_\sigma$ in \eqref{eq:proptau} to obtain
\begin{align}
\mathcal{E}_{\vol}[\phi_\sigma]&=\int_\Omega W\big(D\phi_\sigma(x)\big) \dd x=\int_{\left\{ |\dist_1| \gs 2\sigma \right\}}+\int_{\left\{ |\dist_1| \ls \sigma \right\}}+\int_{\left\{ \sigma < |\dist_1| < 2\sigma \right\}}\\
&\ls \mathcal{E}_{\vol}[\phi] + \int_{\N_\sigma \M_1} |D\phi_\sigma(x)|^p \dd x + \int_{\left\{ \sigma< |\dist_1| < 2\sigma \right\}} |D\phi_\sigma(x)|^p \dd x \\
&\ls \mathcal{E}_{\vol}[\phi] + C\sigma + \int_{\left\{ \sigma< |\dist_1| < 2\sigma \right\}} |D\phi_\sigma(x)|^p \dd x \\
&\ls \mathcal{E}_{\vol}[\phi] +C\sigma + C\int_{\left\{ \sigma< |\dist_1| < 2\sigma \right\}}|D\phi(x)|^p \dd x\\\label{eq:layerestimate}
&\qquad\qquad+C\sigma^{-p}\int_{\left\{ \sigma< |\dist_1| < 2\sigma \right\}} \dist_1(x)^p\, \Big|\,\n_2\Big(\phi\big(\pi_{\M_ 1}(x)\big)\Big)\Big|^p \dd x\\
&\qquad\qquad+ C\sigma^{-p}\int_{\left\{ \sigma< |\dist_1| < 2\sigma \right\}}\Big|\phi\big(\pi_{\M_ 1}(x)\big)-\phi(x)\Big|^p \dd x,
\end{align}
where for the last inequality the product rule for $\tau_\sigma$ and $\phi$ was used, and also that whenever $x=\pi_{\M_ 1}(x)+t\n_1(\pi_{\M_ 1}(x))$, then $t=\big(x - \pi_{\M_ 1}(x)\big)\cdot \n_1\big(\pi_{\M_ 1}(x)\big)=\dist_1(x)$. This implies that the penultimate term of \eqref{eq:layerestimate} tends to zero, since
\[\int_{\left\{ \sigma< |\dist_1| < 2\sigma \right\}} \dist_1(x)^p \,\Big|\,\n_2\Big(\phi\big(\pi_{\M_ 1}(x)\big)\Big)\Big|^p \dd x = \int_{\left\{ \sigma< |\dist_1| < 2\sigma \right\}} \dist_1(x)^p \dd x \ls C \sigma^{p+1}\]
For the last term of \eqref{eq:layerestimate}, noticing that the integrand vanishes at $\M_1$ we use a Poincar\'e inequality for the derivative in the normal direction $\n_1$ on the sets $\left\{ -2\sigma < \dist_1 < 0 \right\}$ and $\left\{ 0 < \dist_1 < 2\sigma \right\}$ (these sets have at least $C^{1,1}$ boundaries since $\sigma$ was chosen small enough, see \cite[Thm.~7.7.1(i)]{DelZol11}) and with optimal constant $C\sigma^p$ to write
\[\begin{aligned}
\sigma^{-p}\int_{\left\{ \sigma < |\dist_1| < 2\sigma \right\}}\Big|\phi\big(\pi_{\M_ 1}(x)\big)-\phi(x)\Big|^p \dd x &\ls \sigma^{-p}\int_{\left\{ 0 < |\dist_1| < 2\sigma \right\}}\Big|\phi\big(\pi_{\M_ 1}(x)\big)-\phi(x)\Big|^p \dd x \\
&\ls C \left\|D\phi\,\big(\n_1 \circ \pi_{\M_ 1}\big)\right\|^p_{L^p(\left\{ 0< |\dist_1| < 2\sigma \right\})} \xrightarrow[\sigma\to 0]{} 0,
\end{aligned}\]
and finally obtain \[\limsup_{\sigma \to 0} \mathcal{E}_{\vol}[\phi_\sigma] \ls \mathcal{E}_{\vol}[\phi].\]

As a remark, let us note that on the one hand the same computations above allow us to prove that \[\|D\phi_\sigma\|_{L^p(\Omega)} - \|D\phi\|_{L^p(\Omega)} \xrightarrow[\sigma\to 0]{} 0,\] while on the other hand $\phi_\sigma$ differs from $\phi$ only on $\N_{2\sigma}\M_1$ while $|\N_{2\sigma}\M_1| \to 0$. From this, up to possibly taking a further subsequence, we conclude that $\phi_\sigma$ converges to $\phi$ not just weakly but also strongly in $W^{1,p}(\Omega; \R^d)$.

\medskip
\noindent \emph{Step 4: Convergence of minimizers.} 
\medskip

As above, since $\theta = 0$ and $\phi \in W^{1,p}_0(\Omega; \R^d)+\Id$, we have
\[\mathcal{E}^\sigma[\phi] \gs C \int_\Omega |D\phi|^p \gs C\left(\|\phi_\sigma\|^p_{W^{1,p}(\Omega)}-1\right),\]
with $C$ independent of $\sigma$. Hence, by the Banach-Alaoglu theorem the $\mathcal{E}^\sigma$ form an equicoercive family of $\Gamma$-converging functionals, which implies \cite[Theorem 1.21]{Bra02} that any sequence $\{\phi^\sigma\}$ of minimizers in $W^{1,p}_0(\Omega; \R^d)+\Id$ of $\mathcal{E}^\sigma$ has a subsequence converging weakly in $W^{1,p}(\Omega; \R^d)$ to a minimizer of $\mathcal{E}^0$. Existence of such minimizers for $\mathcal{E}^\sigma$ can be proved by analogous methods as those used in Theorem \ref{thm:existence}, where Lemma \ref{lem:distest} is modified as in the first step. 
\end{proof}

\begin{rem}[Natural boundary conditions] 
In contrast to the situation in Theorem \ref{thm:existence} where we use global topological properties that are in general only true with fixed Dirichlet boundary, the restriction to $W^{1,p}_0(\Omega; \R^d)+\Id$ in the definition of $\mathcal{E}^\sigma$ and Theorem \ref{thm:gammaconv} is not essential. For the analogue with zero Neumann boundary conditions, the only difference is that $\mathcal{E}^\sigma$ needs to be coercive in $W^{1,p}(\Omega; \R^d)$ as well; this is proved in \cite[Cor.~4.3]{IglRumSch18} using $W(A)\gs C|A|^p$ and the form of $\mathcal{E}_{\match}^\sigma$.
\end{rem}

\begin{rem}[The case $\theta =1$]
From a modelling perspective, the desired scaling for our model is one in which the influence of the volume term vanishes, which is the case when $\theta=1$, a regime in which we could also prove existence of minimizers for all $\sigma >0$ even for the symmetric energy. In that case, the volume energy still determines the values of minimizers outside the narrow band, since it is the only active term there. The same proof of Theorem \ref{thm:gammaconv} shows that in case $\theta = 1$, the functionals $\mathcal{E}^\sigma$ restricted to a set of bounded norm (e.g. $\{\phi \in W^{1,p}_0(\Omega; \R^d)+\Id \,|\, \|D\phi\|_{L^p(\Omega)} \ls C\}$) also $\Gamma$-converge with respect to either $L^p(\Omega; \R^d)$ convergence or the weak $W^{1,p}(\Omega; \R^d)$ topology (which is metrizable on bounded sets) to the surface energy 
\[\int_{\M_1}W( \Dtt\phi )+W\big( \Lambda[D\phi, \S_1, \S_2 \circ \phi, \n_1, \n_2 \circ \phi] \big)\dd \mathcal{H}^{d-1}.\]
This constraint cannot be removed: without it \eqref{eq:stretchedconv} is not guaranteed, since $W(\Dtt \phi)$ is not coercive with respect to derivatives in the normal direction. 
\end{rem}

\begin{rem}[Other choices of membrane energy]
Had we chosen to use for the surface deformation energy (instead of the energies based on the projected derivative $\Dtt \phi$\,) a ``hardened'' but isotropic term depending on the full derivative, of the type
\begin{equation}\label{eq:hardenedEnergy}\int_\Omega \eta_\sigma(\dist_1(x))W\big(D\phi(x)\big)\,\dd x,\end{equation}
we would obtain a $\Gamma$-limit with an integral representation through a density that contains a nontrivial quasiconvexification, and vanishes for matrices whose singular values are less than or equal to $1$ \cite[Theorem 10]{LedRao95}. In consequence, sequences of minimizers of the analogue of $\mathcal{E}^\sigma$ with $\mathcal{E}_\mem^\sigma$ replaced by \eqref{eq:hardenedEnergy} may develop oscillations as $\sigma \to 0$, and the limit functional would not penalize compression of $\M_1$. In contrast, Theorem \ref{thm:gammaconv} (as reflected in \eqref{eq:membraneDensity}, in particular) shows that our projected tangential derivative construction is preserved in the membrane limit, avoiding these drawbacks.

Furthermore, in \cite[Sec.~4.1, Fig.~5]{IglRumSch18} it was demonstrated that using a tangential strain tensor through
\[\int_\Omega \eta_\sigma(\dist_1(x))W\Big(\big[D\phi(x)\P_1(x)\big]^\T\big[D\phi(x)\P_1(x)\big]+\n_1(x) \otimes \n_1(x)\Big)\,\dd x\]
is also not desirable, since this term is not lower semicontinuous and again encourages oscillations in minimizing sequences, even at fixed $\sigma >0$.
\end{rem}

\begin{rem}[Symmetric energies]\label{rem:whynonsym}
A series of papers by O. Anza Hafsa and J.-P. Mandallena (see the overview \cite{AnzMan12} and references therein) tackle the membrane limit with non-interpenetration and orientation preservation conditions. It would be tempting to think of applying this kind of results to attempt to take the limit of the symmetric energies. However in our framework, the surface energies should not enforce orientation preservation since $\det(\Dtt \phi)$ could be negative depending on the relative position of $\M_2$ and $\phi \circ \M_1$, as remarked in Section \ref{subsec:tangential}.

The obstruction for proving Theorem \ref{thm:gammaconv} for the symmetric energies $E^\sigma$ is rather the blending argument with the cutoff function $\tau_\sigma$ used to construct a recovery sequence. What would be needed is a result on approximation of Sobolev homeomorphisms by diffeomorphisms, done in such a way that the corresponding energies converge. Notice that since the energy density is unbounded as the determinant vanishes, this property is not guaranteed by strong convergence. Alternatively, a proof by density is also possible, and a sufficient condition would be approximation by smoother functions with convergence in $L^p$ norm for the derivatives of the inverse transformation, as obtained for planar bi-Lipschitz maps in \cite{DanPra14}. At the time of writing, the existence of such an approximation procedure seems to be an open problem both for planar maps in $W^{1,p}$, $1 < p < +\infty$, and for all three-dimensional cases (\cite{DanPra14}, \cite[Questions 3 and 4]{IwaKovOnn11}, \cite[Open problem 16]{HenKos14}). As noted in \cite{Bal10}, such a result would have deep implications for the mathematical theory of elasticity.
\end{rem}
\section{Computational results for symmetric energies}\label{sec:numerics}
\subsection{Numerical setup}

As in \cite{IglRumSch18}, we have used a `discretize, then optimize' strategy on adaptive hierarchichal quadtree or octree grids defined on $\Omega = (0,1)^d$ with $d=2,3$, coupled with a multiscale first order descent, implemented in the in the Quocmesh library \cite{Quoc}. This means that the solution at one grid, computed through a conjugate gradient method computed with a weighted $H^1$ metric  coupled with Armijo line search, is interpolated into the next finer one and used as an initial condition to continue the descent on the new grid.

\begin{figure*}[ht]
  \centering
  \mbox{}
  \includegraphics[width=.21\linewidth]{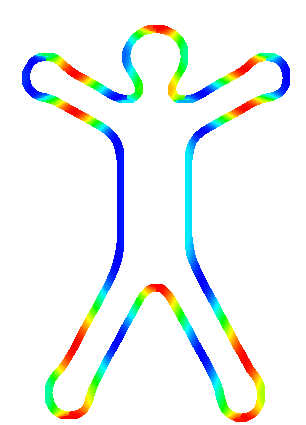}
  \hfill
  \includegraphics[width=.21\linewidth]{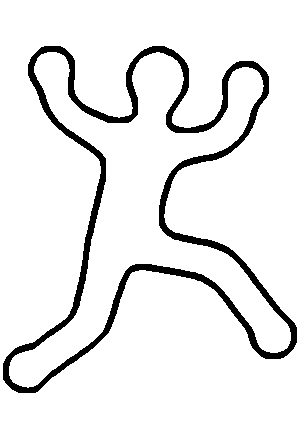}
  \hfill
  \includegraphics[width=.21\linewidth]{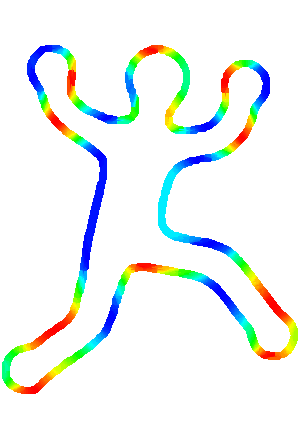}
    \hfill
  \includegraphics[width=.21\linewidth]{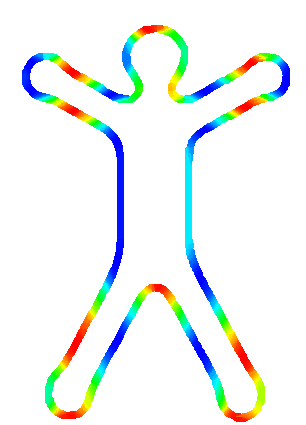}
  \hfill \\
  \mbox{} \hfill
  \includegraphics[width=.32\linewidth]{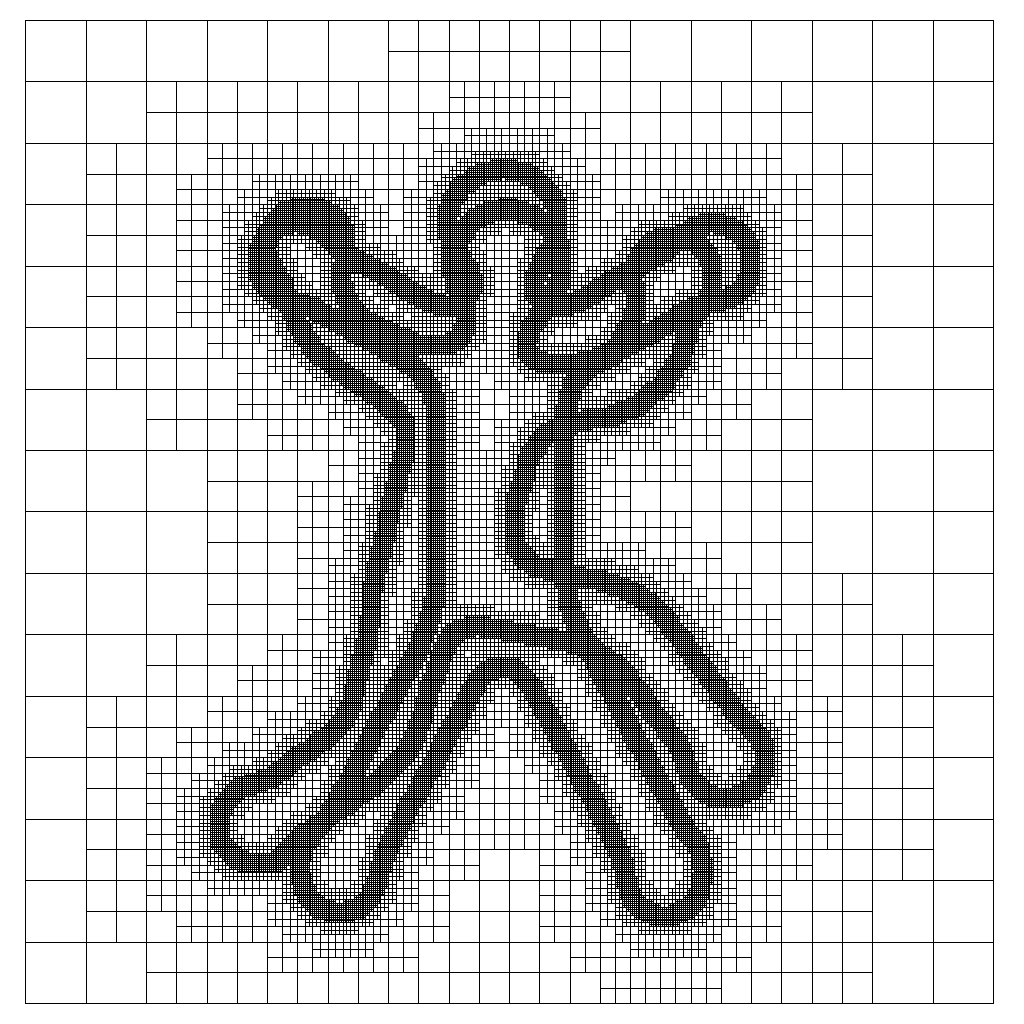}
  \hfill
  \includegraphics[width=.32\linewidth]{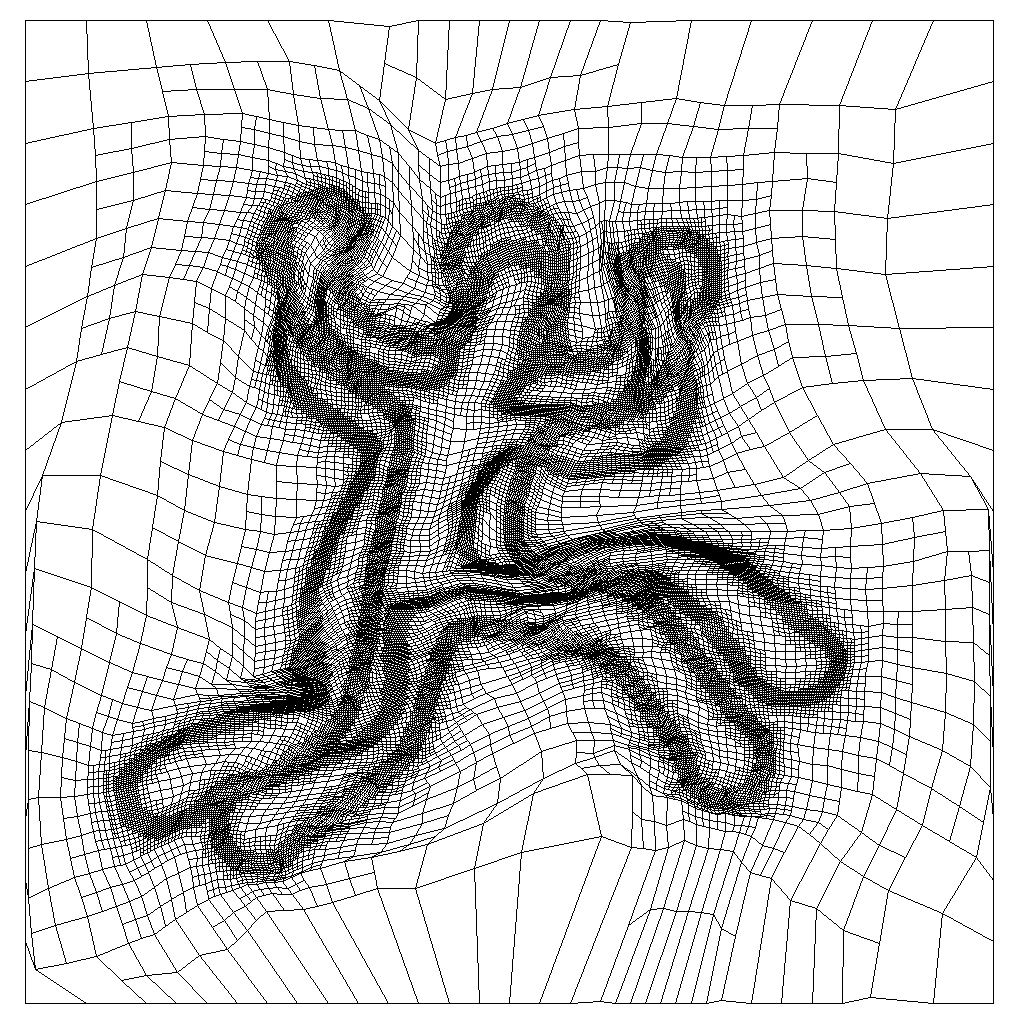}
  \hfill
  \includegraphics[width=.32\linewidth]{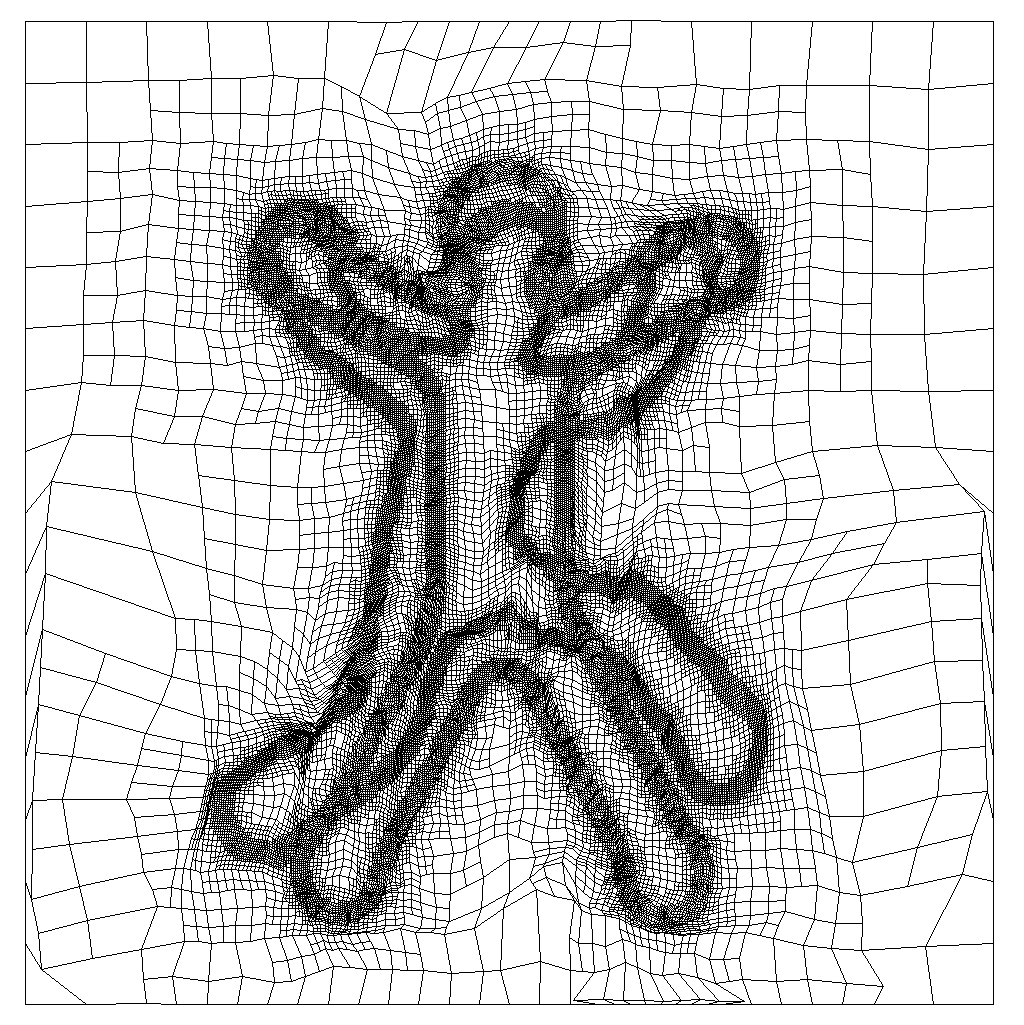}
  \hfill \mbox{}  
  \mbox{}
  \caption{\label{fig:jump}Upper row left to right: Jump shape template $\M_1$ with visualization pattern and target $\M_2$, deformed shape $\phi(\M_1)$ after level $9$ matching with the symmetric energy $E^\sigma$ in \eqref{eq:Esym}, deformed shape $\psi \circ \phi(\M_1)$ after subsequently applying level $9$ of the matching with switched data. Lower row: quadtree grid used with $h_{\min}=2^{-9}$, after applying the direct matching $\phi$, and after applying both matchings through $\psi \circ \phi$.\\ Being able to perfectly numerically realize the symmetry would result in identical leftmost and rightmost images. Although the first and last colored shapes look quite similar, some differences can be seen. For example the red patch on top of the head shifts slightly to the left, an error which can also be easily spotted in the rightmost deformed grid.}
\end{figure*}

The grids are refined around the input shapes $\M_1$ and $\M_2$, to add detail to the main area of interest and maintain accuracy in the coefficients depending on the initial and deformed configuration respectively. The hierarchical structure of the grids allows to search them efficiently (further details are given in Sec. 5 of \cite{IglRumSch18}), which is crucial in our case since the coefficients strongly depend on the deformed configuration. Below, when speaking about these grids, we refer to them as having level $\ell$ when the side of the finest elements present in it is $h_{\ell} = 2^{-\ell}$. Our implementation accepts input shapes given either as triangular meshes in 3D or polygonal curves in 2D, and the distance functions $\dist_i$ are generated through a straightforward modification of the fast marching method \cite{Set99}, taking advantage of the fact that the grids used are subgrids of a regular cartesian grid. 

A straightforward choice of discretization would be to use multilinear finite elements on the squares or cubes contained in the grid, which is the approach used in \cite{IglBerRumSch13} and \cite{IglRumSch18}. However, this type of discretization has some limitations for our application. The main concern is maintaining the deformations injective. On the one hand the Jacobian determinants that appear numerically (that is, on quadrature points) can be enforced to be positive along the descent by using infinite values of the energy and adequate line search for the descent. However, when refining the grid and interpolating the deformation to the newly created elements, this property might be lost: injectivity of a trilinear transformation on a hexahedral element is not even known to be checkable through simple algebraic conditions \cite{KnaKorSum03}. This means that even if the Jacobian determinants are positive at every quadrature point of the original grid, they might not necessarily be positive at all those of the refined grid, a situation which prevents the multiscale descent from continuing after the refinement. This problem occurs only for very small determinant values (`thin' deformed elements) and therefore it can often be avoided, but without guarantees, by keeping the influence of $E_{\vol}$ relatively high.

In fact, this problem can be completely avoided by splitting each square or cube of the grid in two regular triangles or six tetrahedra respectively, and using linear finite elements on the resulting simplices instead. In this way, the gradients are piecewise constant, and since the elements of the subdivided grid are always completely contained in a coarse element, the Jacobian determinant is preserved when interpolating to the refined grid. This has allowed us to eliminate the mentioned problem with negative determinants, and to emulate the regime $\theta = 1$ by decreasing the influence of the volume term with each refinement. Indeed, in the numerical examples presented we have chosen $\sigma = 2h_{\ell}$ and a coefficient for the volume energy proportional to $\sigma$.

Another difference is that since we focus in symmetry and invertibility, Dirichlet conditions fixing the deformations at the boundary to be the identity have been used. In consequence, the size of the shapes compared with that of the domains should be relatively small so that the fixed boundary values do not affect the matching too much through the volume regularization term. This drawback is mitigated by the use of adaptive grids, since these are only refined around the shapes themselves.

Our implementation of the energy and its derivatives follows the formulas in Lemma \ref{lem:polyconvexity} to minimize the appearance of terms related to $\det(D\phi^{-1})$, which have the potential to introduce large numerical errors when injectivity of the deformations is nearly lost.

\begin{figure*}[ht]
  \centering
  \mbox{} \hfill
  \includegraphics[width=.23\linewidth]{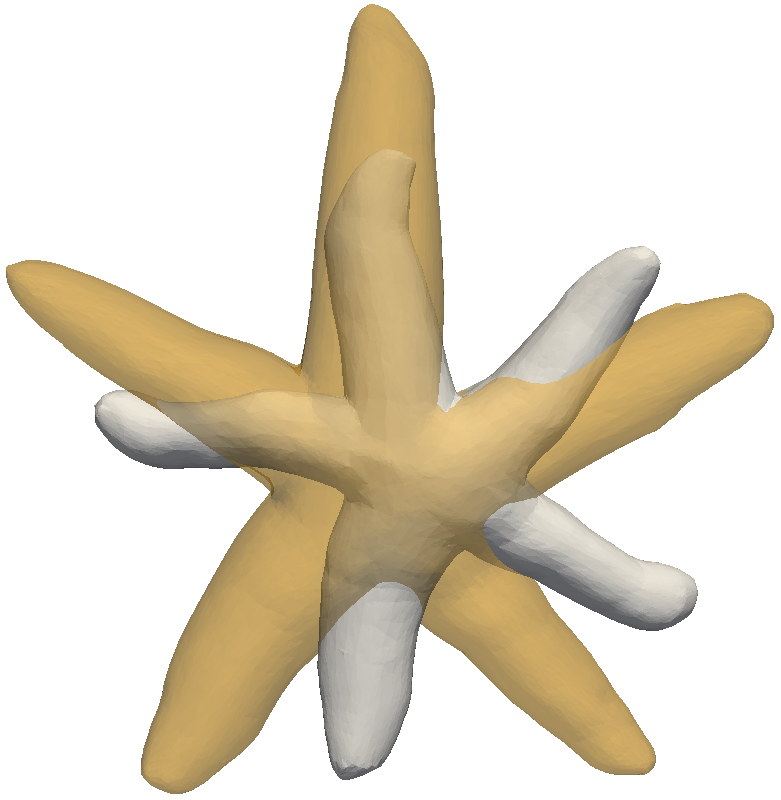}
  \hfill
  \includegraphics[width=.23\linewidth]{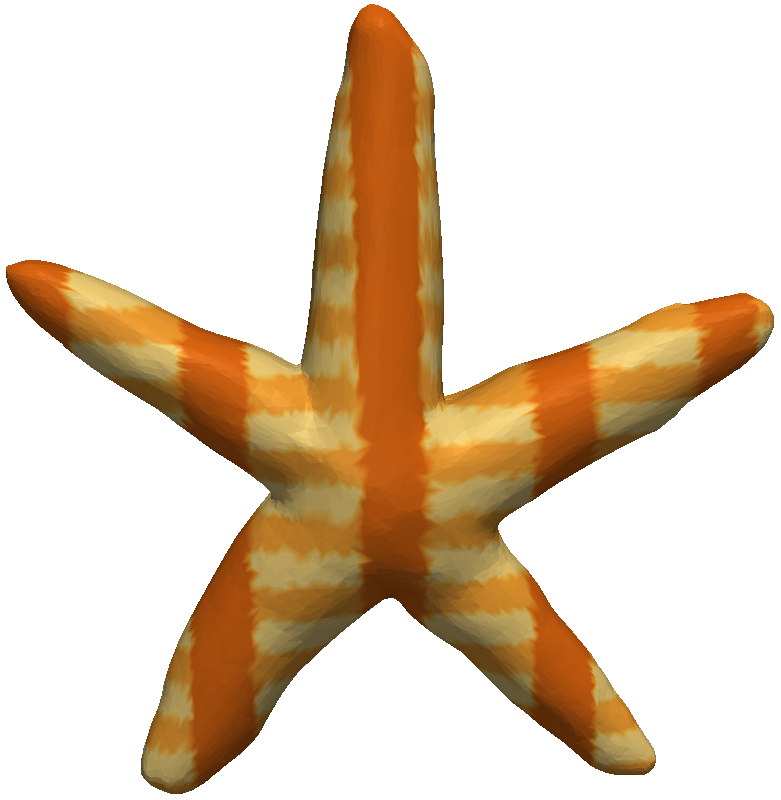}
  \hfill
  \includegraphics[width=.23\linewidth]{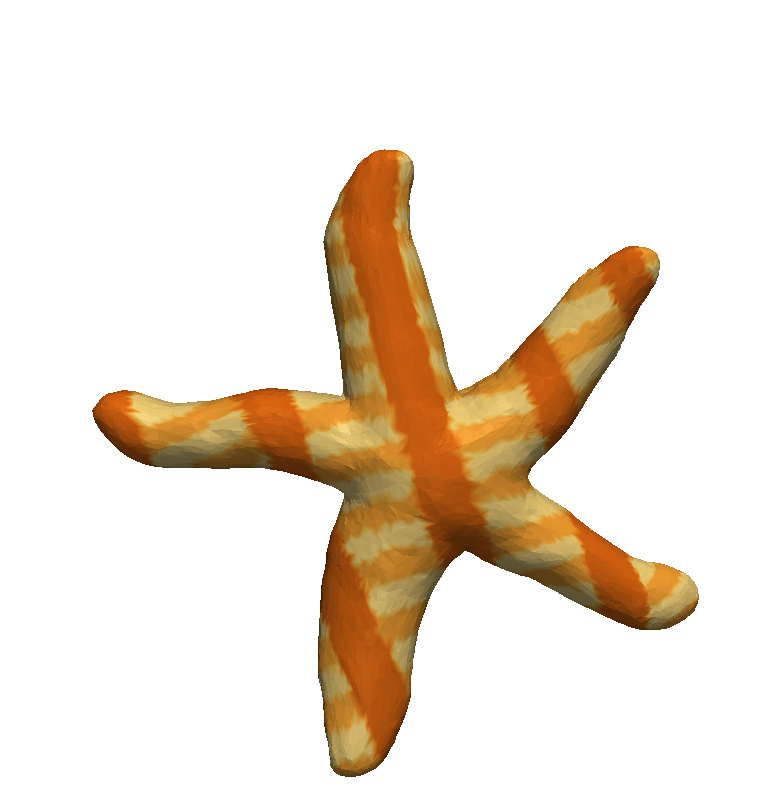}
  \hfill
  \includegraphics[width=.23\linewidth]{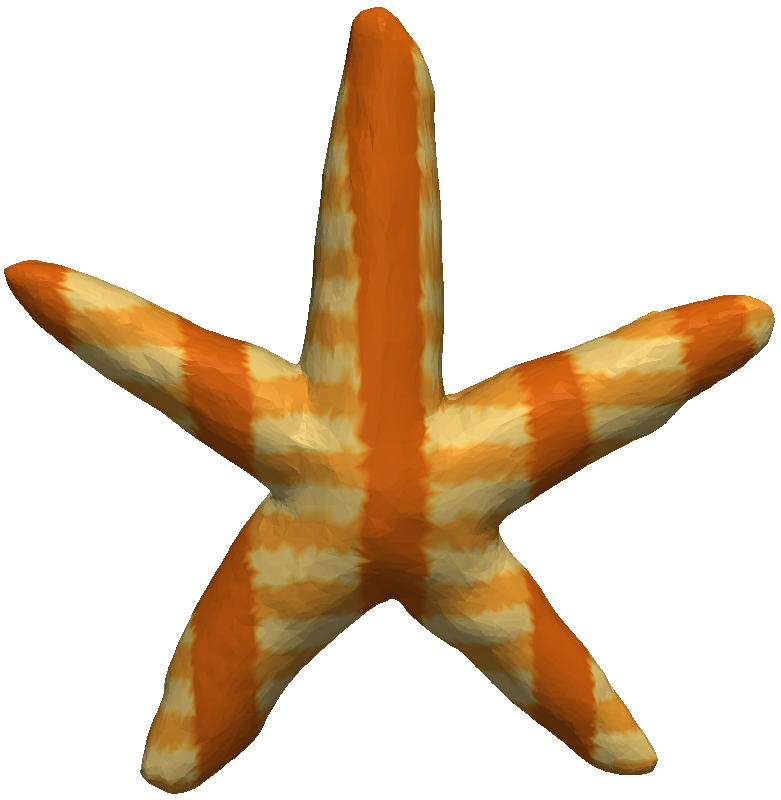}
  \hfill \mbox{}
  \caption{\label{fig:starfish}From left to right: Starfish shapes $\M_1$ (orange) and $\M_2$ (white), textured $\M_1$, deformed shape $\phi(\M_1)$ after level $8$ matching with symmetric energy functional $E^\sigma$, deformed shape $(\psi \circ \phi)(\M_1)$ after subsequently applying level $8$ of the matching with switched data.}
\end{figure*}

\begin{figure*}[ht]
  \centering
  \mbox{} \hfill
  \includegraphics[width=.19\linewidth]{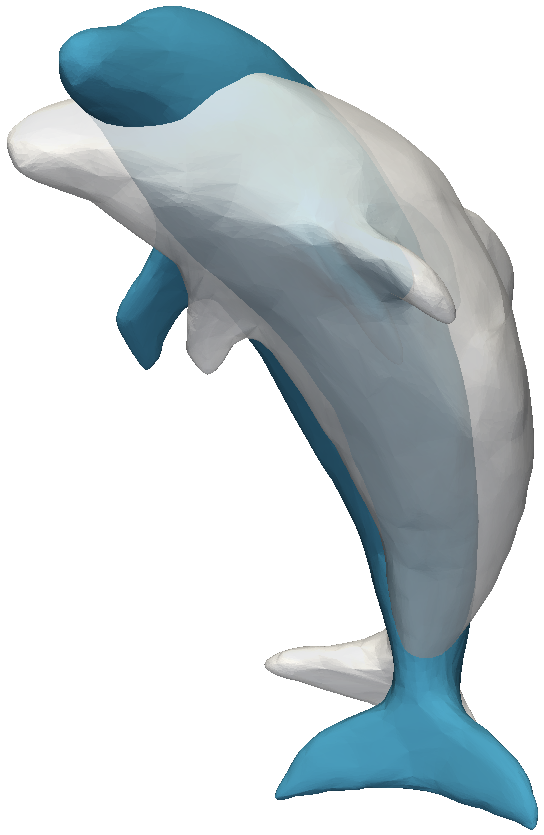}
  \hfill
  \includegraphics[width=.19\linewidth]{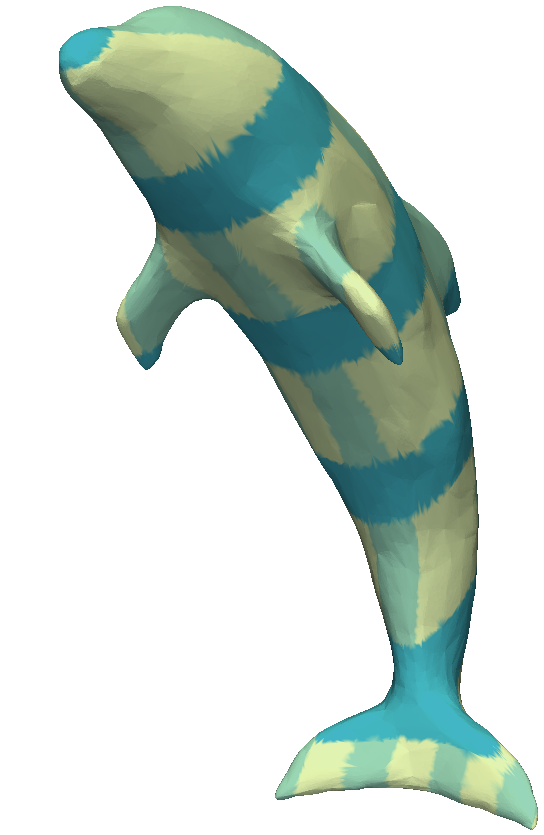}
  \hfill
  \includegraphics[width=.19\linewidth]{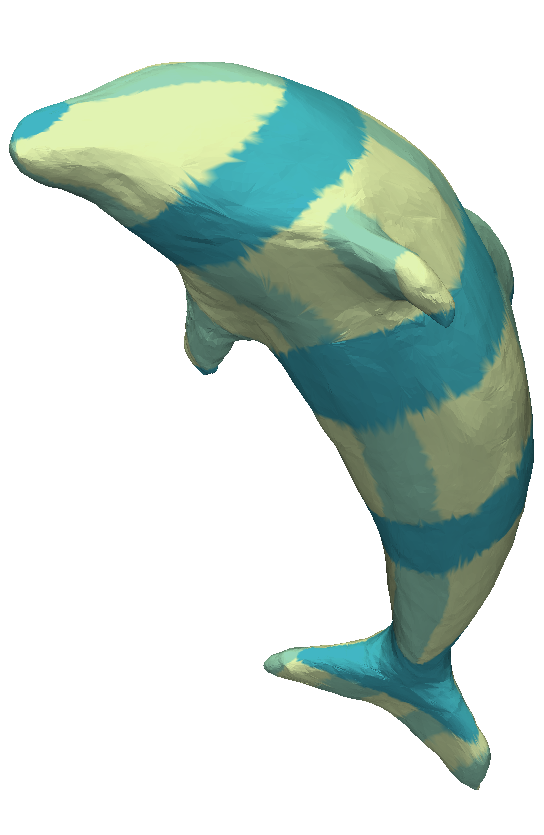}
  \hfill
  \includegraphics[width=.19\linewidth]{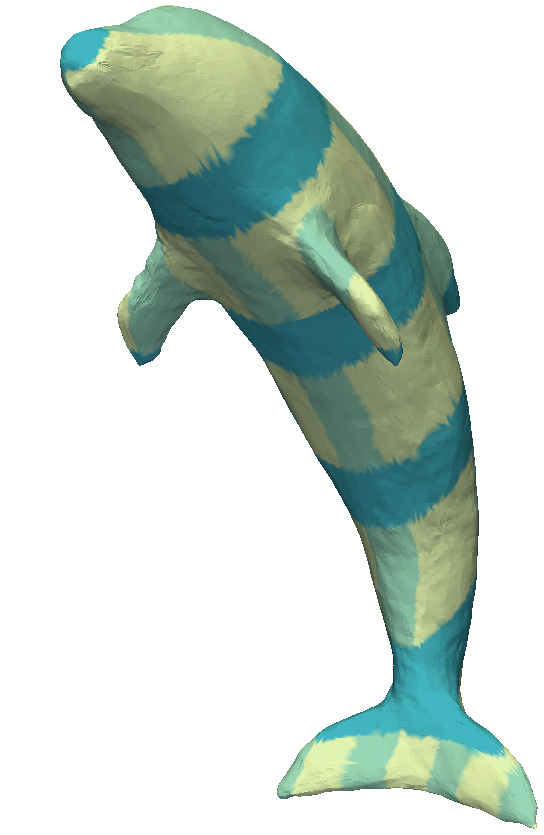}
  \hfill \mbox{}
  \caption{\label{fig:dolphin}From left to right: Dolphin shapes $\M_1$ (blue) and $\M_2$ (white), textured $\M_1$, deformed shape $\phi(\M_1)$ after level $8$ matching with symmetric energy functional $E^\sigma$, deformed shape $(\psi \circ \phi)(\M_1)$ after subsequently applying level $8$ of the matching with switched data.}
\end{figure*}

\subsection{Symmetry in the numerical results}

We have computed several examples both with the novel symmetric energy $E^\sigma$, and with a comparison energy defined only on the direct transformation, but with a volume term that ensures injectivity. Indeed, invertibility of the obtained deformations is required to perform the comparisons in the form proposed. The energy that we compare against is closely related to $\mathcal{E}^\sigma$ of Section \ref{sec:limits} and the one formulated in \cite{IglRumSch18}. It reads
\begin{equation}\label{eq:Ecomparison}
\overline{\mathcal{E}}^\sigma[\phi]:=\mathcal{E}^\sigma_{\match}[\phi]+\mathcal{E}^\sigma_{\mem}[\phi]+\mathcal{E}^\sigma_{\bend}[\phi]+\overline{\mathcal{E}}^ \sigma_{\vol}[\phi], \text{ where }
\end{equation}
\begin{equation*}\label{eq:EcomparisonVol}
\overline{\mathcal{E}}_{\vol}^\sigma[\phi]:=\sigma^\theta\int_\Omega W_{\vol}(D\phi), \text{ with }\begin{cases}W_{\vol}(A)=|A|^3+|\cof A|^3 + 3 (\det A)^{-2}&\text{if }d=3\\W_{\vol}(A)=|A|^2+(\det A)^{-2}&\text{if }d=2,\end{cases}
\end{equation*}
where $\mathcal{E}^\sigma_{\match}, \mathcal{E}^\sigma_{\mem}, \mathcal{E}^\sigma_{\bend}$ are the expressions in \eqref{eq:Edir}, and using the polyconvex density $W$ defined in \eqref{eq:Wbounded}. It can be directly checked that the identity matrix $\IdM \in \R^{d\times d}$ is stationary for $W_{\vol}$ by writing it in terms of singular values. 

The parameters used were identical for both energies and a given shape, as listed in Table \ref{tab:parameters}, with the exception of the different volume density in $\overline{\mathcal{E}}_{\vol}$, but with each volume energy multiplied with the same coefficient $c_\vol$. The energy density $W$ used for all terms of $E^\sigma$ that require it was the one introduced in \eqref{eq:Wbounded}, and we used $\theta=1$ and $q=p=d+1$ replicating the regime analyzed in Section \ref{sec:existence}.

\begin{table}[ht]
\centering
    \begin{tabular}{| l | l | l | l | l | l | l | l | l |}
    \hline
  & $c_{\match}$ & $c_{\vol}$ & $c_{\mem}$ & $c_{\bend}$ & $\sigma$ & $q$ & $\theta$ & $\ell_{\min}, \ell_{\max}$ \\ \hline
Dolphin   & 4.096 & 0.8 & 1.0 & 0.2 & $2^{-\ell + 1}$ & 4 & 1 & 4,8  \\ \hline
Starfish  & 4.096 & 0.8 & 1.0 & 0.2 & $2^{-\ell + 1}$ & 4 & 1 & 4,8  \\ \hline
Jump      & 0.512 & 0.8 & 1.0 & 1.0 & $2^{-\ell + 1}$ & 3 & 1 & 4,9  \\ \hline
    \end{tabular}
    \caption{Parameters used for the numerical examples, where $c_{\match}$, $c_{\vol}$, $c_{\mem}$ and $c_{\bend}$ are multiplicative factors for the corresponding terms of \eqref{eq:Esym} or \eqref{eq:Ecomparison}.}
    \label{tab:parameters}
\end{table}

It is important to notice that, although the energy is symmetric with respect to switching the shapes and taking the inverse of the deformations, the gradient descent procedure is not. Therefore, in practice perfect symmetry can not be expected in the numerical results, and the extent to which it appears depends on not ending up in different local minima, and how closely these minima are approximated by the computation. In any case our numerical experiments show a marked improvement towards symmetry. 

Figure \ref{fig:jump} shows a 2D example of a shape $\M_1$ undergoing first the deformation $\phi$ computed using $E^\sigma$ to match $\M_1$ to $\M_2$, then also the one $\psi$ with switched inputs matching $\M_2$ to $\M_1$, and the corresponding deformed grids. In Figures \ref{fig:starfish} and \ref{fig:dolphin} analogous 3D examples are shown. In each of these cases, being able to exactly realize the symmetry property numerically would result in identical shapes and grids before any deformation and after applying both. 

\begin{figure}[ht]
  \centering
  \mbox{} \hfill
  \includegraphics[width=.38\linewidth]{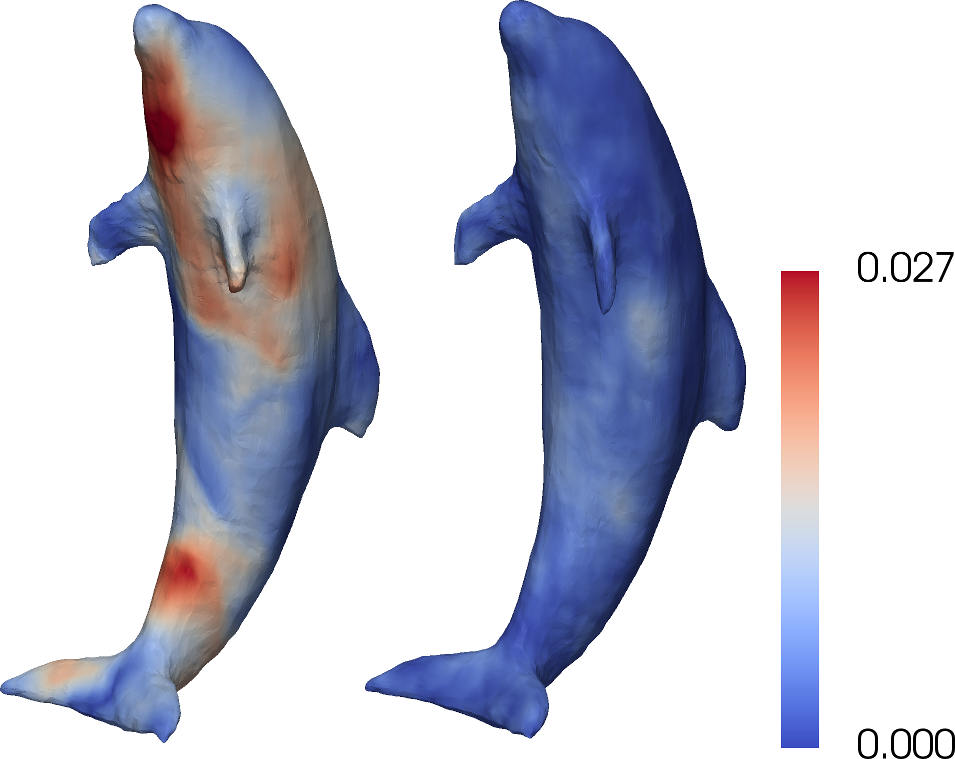}
  \hfill
  \raisebox{.0\height}{\includegraphics[width=.56\linewidth]{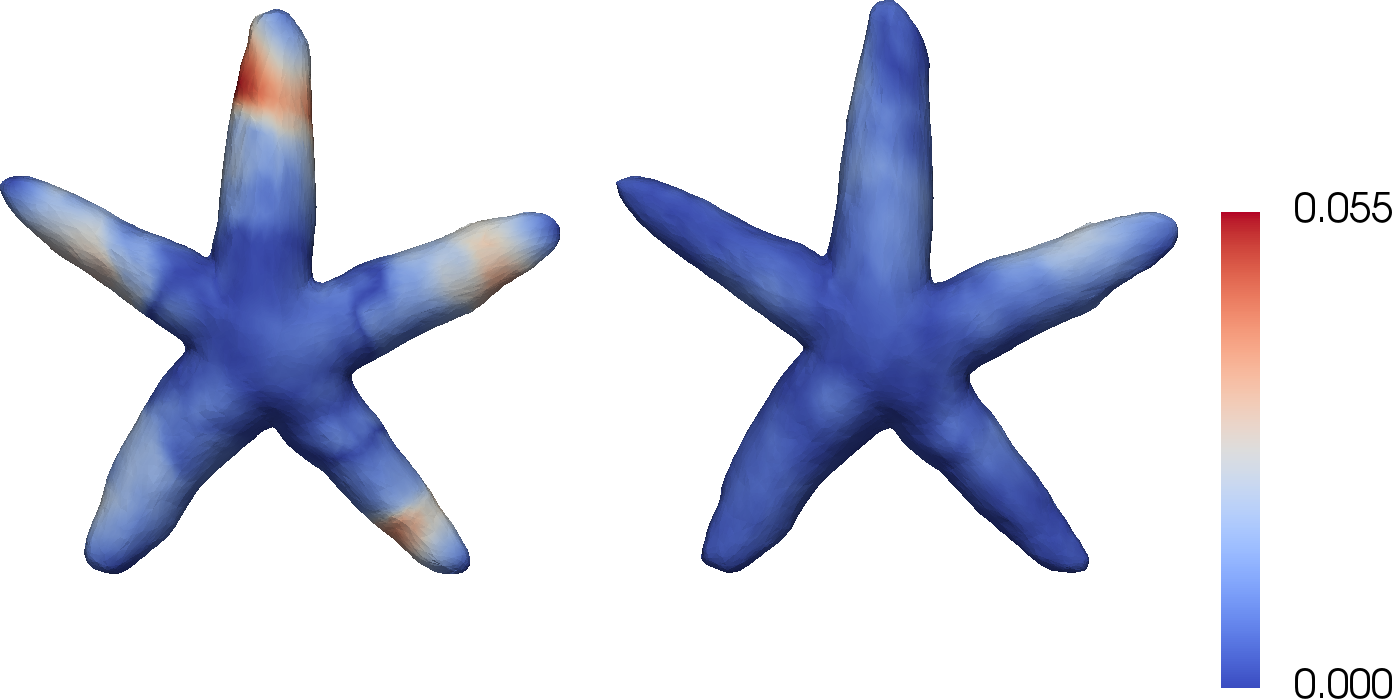}}
  \hfill \mbox{}
    \caption{Pointwise norm of the residual displacement $|\psi \circ \phi - \Id|$, plotted as texture over $( \psi \circ \phi )( \M_1 )$. Left shape for each case: Result with energy $\overline{\mathcal{E}}^\sigma$ not taking into account the inverse, corresponding to (dir) case in Table \ref{tab:inverseError}. Right shape: Result with the symmetric energy $E^\sigma$ (same parameters and descent procedure), corresponding to (sym) case in Table \ref{tab:inverseError}. \\ As expected, only subtle differences appear in the shapes themselves. Most of the erroneous displacement on the surface occurs tangentially and in zones where the largest bending takes place (cf. Figures \ref{fig:starfish} and \ref{fig:dolphin}).}
    \label{fig:inverseError}
\end{figure}

\begin{table}[ht]\centering
    \begin{tabular}{| l | l | l | l | l |}
    \hline
    Case & $\| \psi \circ \phi - \Id \|_{L^2(\Omega)}$ & $\| \psi \circ \phi - \Id \|_{L^\infty(\Omega)}$  & $\text{avg}(| \psi \circ \phi - \Id |, \M_1)$ & $\| \psi \circ \phi - \Id \|_{L^\infty(\M_1)}$ \\ \hline
\text{D, dir} & 0.0299  & 0.0583 & 0.0108  & 0.0570  \\ \hline 
\text{D, sym} & 0.0271  & 0.0561 & 0.00281 & 0.0114  \\ \hline 
\text{S, dir} & 0.0637  & 0.136  & 0.0132  & 0.0546  \\ \hline 
\text{S, sym} & 0.0473  & 0.115  & 0.00570 & 0.0240  \\ \hline 
\text{J, dir} & 0.0715  & 0.141  & 0.0223  & 0.141   \\ \hline 
\text{J, sym} & 0.0419 & 0.0982 & 0.00737 & 0.0982  \\ \hline 
    \end{tabular}
    \caption{Average and maximum norm of the residual displacement $|\psi \circ \phi - \Id|$ at last computation level for the dolphin (D), starfish (S) and jump (J) examples, computed with the energy $\overline{\mathcal{E}}^\sigma$ that penalizes only the direct transformation \eqref{eq:Ecomparison}, (dir) and with the new symmetric energy $E^\sigma$ \eqref{eq:Esym}, (sym).\\
On average larger errors are seen outside the shapes themselves, which is consistent with the decreasing influence of the volume term over the refinements to recreate the regime $\theta=1$.}
    \label{tab:inverseError}
\end{table}
In Figure \ref{fig:inverseError} and Table \ref{tab:inverseError} we quantify the failure of symmetry in these examples by evaluating the distance $|\psi \circ \phi - \Id|$ between the identity and composition of the deformations matching the shapes in opposite orders, when using the symmetric energy $E^\sigma$ and the non-symmetric energy $\overline{\mathcal{E}}^\sigma$ as comparison. Averages on $\M_1$ are computed from evaluation of the finite element functions on the vertices of the triangular meshes or polygons used as input and sub-grid initialization of the fast marching method to compute $\dist_i$, with equal weights for all such points. This avoids having to integrate numerically discrete functions defined on $\Omega$, for which the surface meshing is not compatible. The input surfaces, which are fairly evenly triangulated, are only used for initializing the computation of $\dist_i$ and not in the computation for $\phi$.

\paragraph*{Acknowledgements.}
This work has been supported by the Austrian Science Fund (FWF) within the national research network `Geometry+Simulation', project S11704.
\bibliographystyle{plain}
\bibliography{Igl17}
\end{document}